\documentclass[12pt, reqno]{amsart}

\usepackage{latexsym}
\usepackage{amssymb}
\usepackage{mathrsfs}
\usepackage{amsmath}
\usepackage{fancybox,color}
\usepackage{enumerate}
\usepackage[latin1]{inputenc}
\usepackage[active]{srcltx}
\usepackage[colorlinks]{hyperref}
\usepackage{hypernat}
\usepackage{eurosym}
\usepackage{graphicx} 

\newcommand{\hide}[1]{}

\newcommand{\tilaa}{} %{\vspace{30 ex}}

\newcommand{\kom}[1]{}
\newcommand{\komt}[1]{} %{\tilaa}
\newcommand{\idea}{\noindent {\textsc {Idea: }}}

 \def\1{\raisebox{2pt}{\rm{$\chi$}}}
\def\a{{\bf a}}

\newtheorem{theorem}{Theorem}[section]

\newtheorem{lemma}[theorem]{Lemma}
\newtheorem{proposition}[theorem]{Proposition}
\theoremstyle{definition}
\newtheorem{definition}[theorem]{Definition}
\newtheorem{example}[theorem]{Example}
\theoremstyle{remark}
\newtheorem{remark}[theorem]{Remark}

\newcommand{\R}{{\mathbb R}}
\newcommand{\RR}{{\mathbb R}}

\newcommand{\N}{{\mathbb N}}

\newcommand{\E}{{\mathbb E\,}}

 \newcommand{\eps}{{\varepsilon}}
 \def\1{\raisebox{2pt}{\rm{$\chi$}}}
 
\newcommand{\abs}[1]{\left|#1\right|}

\newcommand{\Rn}{\mathbb{R}^n}

\usepackage{pgf,tikz}\usepackage{mathrsfs}\usetikzlibrary{arrows}
\usetikzlibrary{decorations.pathreplacing, calligraphy} %for big curly bracket
\def\vint_#1{\mathchoice%
          {\mathop{\kern 0.2em\vrule width 0.6em height 0.69678ex depth -0.58065ex
                  \kern -0.8em \intop}\nolimits_{\kern -0.4em#1}}%
          {\mathop{\kern 0.1em\vrule width 0.5em height 0.69678ex depth -0.60387ex
                  \kern -0.6em \intop}\nolimits_{#1}}%
          {\mathop{\kern 0.1em\vrule width 0.5em height 0.69678ex
              depth -0.60387ex
                  \kern -0.6em \intop}\nolimits_{#1}}%
          {\mathop{\kern 0.1em\vrule width 0.5em height 0.69678ex depth -0.60387ex
                  \kern -0.6em \intop}\nolimits_{#1}}}
\def\vintslides_#1{\mathchoice%
          {\mathop{\kern 0.1em\vrule width 0.5em height 0.697ex depth -0.581ex
                  \kern -0.6em \intop}\nolimits_{\kern -0.4em#1}}%
          {\mathop{\kern 0.1em\vrule width 0.3em height 0.697ex depth -0.604ex
                  \kern -0.4em \intop}\nolimits_{#1}}%
          {\mathop{\kern 0.1em\vrule width 0.3em height 0.697ex depth -0.604ex
                  \kern -0.4em \intop}\nolimits_{#1}}%
          {\mathop{\kern 0.1em\vrule width 0.3em height 0.697ex depth -0.604ex
                  \kern -0.4em \intop}\nolimits_{#1}}}

\newcommand{\kint}{\vint}
\newcommand{\intav}{\vint}
\newcommand{\aveint}[2]{\mathchoice%
          {\mathop{\kern 0.2em\vrule width 0.6em height 0.69678ex depth -0.58065ex
                  \kern -0.8em \intop}\nolimits_{\kern -0.45em#1}^{#2}}%
          {\mathop{\kern 0.1em\vrule width 0.5em height 0.69678ex depth -0.60387ex
                  \kern -0.6em \intop}\nolimits_{#1}^{#2}}%
          {\mathop{\kern 0.1em\vrule width 0.5em height 0.69678ex depth -0.60387ex
                  \kern -0.6em \intop}\nolimits_{#1}^{#2}}%
          {\mathop{\kern 0.1em\vrule width 0.5em height 0.69678ex depth -0.60387ex
                  \kern -0.6em \intop}\nolimits_{#1}^{#2}}}
\newcommand{\ud}{\, d}
\newcommand{\half}{{\frac{1}{2}}}

\newcommand{\ol}{\overline}
\newcommand{\Om}{\Omega}
\newcommand{\I}{\textrm{I}}
\newcommand{\II}{\textrm{II}}
\newcommand{\dist}{\operatorname{dist}}
\renewcommand{\P}{\mathbb{P\,}}

\newcommand{\om}{\omega}
\newcommand{\trm}{\textrm}

\newcommand{\vp}{\varphi}

\newcommand{\divt}{\operatorname{div}}

\newcommand{\tr}{\operatorname{tr}}
\renewcommand{\a}{\alpha}

\begin{document}

\title[]{Notes on tug-of-war games and the $p$-Laplace equation}

\author[]{Mikko Parviainen}
\address{Department of Mathematics and Statistics, University of
Jyv\"askyl\"a, PO~Box~35, FI-40014 Jyv\"askyl\"a, Finland}
\email{mikko.j.parviainen@jyu.fi}

%\author[]{}
%\address{}
%\email{}

\date{\today}
\keywords{} \subjclass[2010]{}

\maketitle

\noindent{\bf Preface:} These notes were written up after my lectures at the Beijing Normal University  in January 2022. 
I am grateful to Yuan Zhou for organizing such an interesting intensive period. I would like to thank \'Angel Arroyo and Jeongmin Han for reading the original manuscript and contributing valuable comments; as well as Janne Taipalus for some of the illustrations.

The objective is the interplay between stochastic processes and partial differential equations. To be more precise, we focus on the connection between the nonlinear $p$-Laplace equation, and the stochastic game called tug-of-war with noise. The connection in this context was discovered roughly 15 years ago, and has provided novel insight and  approaches ever since. These  lecture notes provide a short introduction to the topic and to more research oriented literature, as well as to the recent monographs of Lewicka, \emph{A Course on Tug-of-War Games with Random Noise}, and Blanc-Rossi, \emph{Game Theory and Partial Differential Equations}. We also introduce the parabolic case side by side with the elliptic one, and cover some parts of the regularity theory. 

\pagebreak

\tableofcontents

\pagebreak

\section{Introduction}

%%%%%%%%%%

The fundamental works of Doob, Hunt, Kakutani, Kolmogorov, L\'evy and many others have shown a profound and powerful connection between the classical linear potential theory and the corresponding probability theory. The idea behind the classical interplay is that harmonic functions and martingales share a common origin in mean value properties, i.e.\ roughly:
\begin{align*}
\text{\mbox{$u$ is harmonic $\sim\ u(x)=\frac{1}{\abs{B_\eps}}\int_{B_\eps(x)} u(y) \ud y  \sim$ random walk,}
}
\end{align*}
where $\sim$ can be read as 'is related to'.
The connection between harmonic functions, the mean value property and the expected value of a random walk is further illustrated in the examples below.
Such a connection appears to require linearity, but the approach turns out to be useful in the nonlinear theory as well.

These notes use basic probability: the reader can consult for example \cite{varadhan01} or \cite{evans13}.

\subsection{First examples}

\begin{example}[Random walk]
\label{ex:1d-random}
Consider a one dimensional discrete random walk, where we toss a coin to decide whether we move left or right. In other words,
when at a point on the $\eps$-grid
\begin{align*}
\{0,\eps,2\eps,\ldots ,1\}\setminus \{0,1\},
\end{align*}
 a token is moved 
\begin{align*}
&\text{one step left with probability } \half, \\
&\text{one step right with probability }\half.
\end{align*}
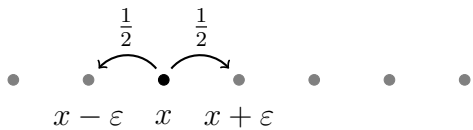
\begin{figure}[h]
 \begin{tikzpicture}
  \filldraw[gray] (-2,0) circle (2pt);
  \filldraw[gray] (-1,0) circle (2pt);
  \filldraw[black] (0,0) circle (2pt);
  \filldraw[gray] (1,0) circle (2pt);
  \filldraw[gray] (2,0) circle (2pt);
  \filldraw[gray] (3,0) circle (2pt);
  \filldraw[gray] (4,0) circle (2pt);
  \draw[thick ,->] (-0.1,0.15) arc (40:140:0.5);
  \draw[thick ,->] (0.1,0.15) arc (140:40:0.5);
  \node at (-0.5,0.7) {$\frac{1}{2}$};
  \node at (0.5,0.7) {$\frac{1}{2}$};
  \node at (-1,-0.5) {$x-\eps$};
    \node at (0,-0.5) {$x$};
  \node at (1,-0.5) {$x+\eps$};
 \end{tikzpicture}
    \caption{Discrete random walk.}
\end{figure}

Throughout the introduction, the step size is of the form
$$\eps=2^{-k} \quad \text{for some}\quad k\in \mathbb N.$$  
\tilaa
When we reach the boundary at 0 or 1, we stop and get a payoff
\begin{align*}
F(0)=0 \text{ or }F(1)=1,
\end{align*}
based on whichever point 0 or 1 we ended up with. 

What is the expected payoff when starting at $x_0$? Let us denote 
\[
\begin{split}
u(x_0):=\mathbb E^{x_0}[F(x_{\tau})],
\end{split}
\]  
where the expectation is taken over all the possible paths (i.e.\ sequence of points visited), $x_{\tau}$ denotes whichever point 0 or 1 we ended up with on the boundary, and $\tau$ is a stopping time denoting the number of the round when we hit the boundary. 
Looking at one step, a dynamic programming principle (DPP) holds for $x\in \{\eps,2\eps,\ldots ,1-\eps\}$:
\[
\begin{split}
\begin{cases}
u(x)=\half(u(x-\eps)+u(x+\eps)) & \\
u(0)=0,\ u(1)=1. & 
\end{cases}
\end{split}
\]
The solution is
\[
\begin{split}
u(i\eps)=i\eps,\qquad  i=0,1,\ldots, \eps^{-1}.
\end{split}
\]
\idea The DPP says that the expectation can be computed by summing up the expectations in the neighboring points with the corresponding step probabilities.

\noindent {\bf Hitting probabilities:} 
In this example, the value $u(x_0)$ describes the probability of hitting $1$ first.
We state without a proof that the process reaches $0$ or $1$ almost surely, i.e.\ with probability 1. It holds that
\begin{align}
\label{eq:hitting-prob}
u(x_0)&=\mathbb E^{x_0}[F(x_{\tau})]=\P^{x_0}(x_{\tau}=1)\cdot 1+\P^{x_0}(x_{\tau}=0)\cdot 0\nonumber \\
&=\P^{x_0}(x_{\tau}=1).
\end{align}
\end{example}

\begin{example}[Running time]
\sloppy
\label{ex:surplus-of-two}
Tossing again a coin, we consider a two players game. We keep track of how many heads and tails occur. Player I wins if she gets a surplus of two in heads (= two more heads) first, and Player II if he gets surplus of two in tails.
 Denote the grid 
 $$
 -2,-1,0,1,2,
$$
where the position denotes the surplus ($-$ for deficit) of heads. The process may be modeled as a random walk, where we jump right with probability $\half$ (got heads) and left with probability $\half$ (got tails).

 What is the remaining expected number of tosses in the game when we start observing in a given point $x$ of the grid?
Denote
 \begin{align*}
u(x)=\mathbb E^{x}['\text{number of tosses}']=\mathbb E^{x}[\tau],
\end{align*}  
where $\tau$ is a stopping time denoting the number of the round when we reach $-2$ or $2$.
The following DPP holds  for $x\in \{-1,0,1\}$ 
\[
\begin{split}
\begin{cases}
 u(x)=\half(u(x-1)+u(x+1))+1 & \\
  u(\pm 2)=0. & 
\end{cases}
\end{split}
\]  
The solution is $u(x)=-x^2+4$ for $x\in \{-2,-1,0,1,2\}$.

\idea The expected number of rounds can be computed by summing up the expected number of rounds in the neighboring points with the corresponding step probabilities plus one step to get there.

In the same way, we may form the DPP for the  expected number of steps in the random walk on the $\eps$-grid of Example \ref{ex:1d-random}:
\begin{align}
\label{eq:dpp-e-steps}
\begin{cases}
 u(x)=\half(u(x-\eps)+u(x+\eps))+1 & \\
  u(0)=0,\quad u(1)=0. & 
\end{cases}
\end{align}
\end{example}

%%%%%%%%%%%%%%%%%%%%%%%%%%%%%%%%%%%%%%%%%%%%%%%
\begin{example}[Related PDEs]
\label{ex:walk-laplace}
Next we modify the DPP
\[
\begin{split}
u(x)=\half(u(x-\eps)+u(x+\eps))
\end{split}
\]
in  Example \ref{ex:1d-random}
and 
get
\begin{align*}
\frac{1}{2\eps^2} (u(x-\eps)-2u(x)+u(x+\eps))=0.
\end{align*}
The left hand side is a well-known discretization of $\half u''$ and thus the random walk seems to be related to the problem
\[
\begin{split}
\begin{cases}
u''=0 & \text{ in }(0,1)  \\
 u(0)=0,\ u(1)=1.& 
\end{cases}
\end{split}
\]

In Example~\ref{ex:surplus-of-two} we calculated the expected number of steps and in the case of the $\eps$-grid ended up with (\ref{eq:dpp-e-steps}).
Now let us instead compute the running time for the walk on the $\eps$-grid in the case when one step takes $ \eps^2$ seconds. Then the DPP takes the form 
\[
\begin{split}
\begin{cases}
u(x)=\half(u(x-\eps)+u(x+\eps))+\eps^2 & \\
u(0)=0,\ u(1)=0. & 
\end{cases}
\end{split}
\]
Rearranging the DPP $u(x)=\half(u(x-\eps)+u(x+\eps))+\eps^2$, we get
\[
\begin{split}
 \frac{1}{\eps^2} (u(x-\eps)-2u(x)+u(x+\eps))=-2.
\end{split}
\]
Thus this seems to be related to the problem
\[
\begin{split}
\begin{cases}
u''=-2 & \text{ in }(0,1)  \\
 u(0)=0,\ u(1)=0.& 
\end{cases}
\end{split}
\]
If $u(x)=\half(u(x-\eps)+u(x+\eps))+ \eps^2$, then 
$$
v(x):=\eps^{-2}u(x)
$$ 
solves 
$$
v(x)=\half(v(x-\eps)+v(x+\eps))+1.
$$ 
This also suggests that the number of steps in the above $\eps$-step random walk has a bound of the form $C/\eps^2$. We will need this observation later.

\noindent {\sc Another point of view:} we may generalize the approach beyond running times. Instead, every step gives $+\eps^2 f(x)$ running payoff/cost and $F$ is the final payoff/cost: 
\[
\begin{split}
\begin{cases}
u(x)=\half(u(x-\eps)+u(x+\eps))+\eps^2 f(x) & \\
u(0)=F(0),\ u(1)=F(1). & 
\end{cases}
\end{split}
\]
This is related to the problem
\begin{align*}
 \begin{cases}
\half u''(x)=-f(x), & \\
u(0)=F(0),&\\
 u(1)=F(1). & 
\end{cases}
\end{align*}
\end{example}

%%%%%%%%%%
\begin{example}[Random walk in 2D]
\label{ex:2d-discrete}
Let $\Om$ be a two dimensional regular grid $\{\eps,2\eps,\ldots ,1-\eps\}\times \{\eps,2\eps,\ldots ,1-\eps\}$ with the boundary $\partial \Om$ defined as $(\{0,\eps,2\eps,\ldots ,1\}\times \{0,\eps,2\eps,\ldots ,1\})\setminus \Om$. When in the domain, we step to one of the four neighbors (up/down/left/right) each with probability $\frac14$ until we hit the boundary, and at the boundary we are  given a payoff  function $F$. Again, denote 
\begin{align*}
u(x_0,y_0):&=\mathbb E^{x_0,y_0}[F(x_{\tau},y_{\tau})]
\end{align*}
where $(x_0,y_0)$ denotes the starting point, and $(x_{\tau},y_{\tau})$ the hitting point at the boundary on the round $\tau$ similarly as before.
Then we have the DPP
\[
\begin{split}
u(x,y)=\frac14(u(x,y+\eps)+u(x,y-\eps)+u(x+\eps,y)+u(x-\eps,y)) 
\end{split}
\]
that can be written as
\[
\begin{split}
0=\frac{1}{4\eps^2} (u(x,y+\eps)+u(x,y-\eps)+u(x+\eps,y)+u(x-\eps,y)-4u(x,y)). 
\end{split}
\]
The right hand side is a two dimensional discretization of the Laplacian apart from the factor 1/4 i.e.\
\begin{align*}
\Delta u&=u_{xx}+u_{yy}\\
&\approx \frac{1}{\eps^2}(u(x,y+\eps)+u(x,y-\eps)+u(x+\eps,y)+u(x-\eps,y)-4u(x,y)).
\end{align*}
Thus the two dimensional random walk seems to be related to the problem
\[
\begin{split}
\begin{cases}
\Delta u=0 & \text{ in }\Om \\
 u=F&\text{ on } \partial \Om. 
\end{cases}
\end{split}
\]
This approach generalizes to $\R^n$ as well.

There is also a continuous version of the random walk that we do not consider in detail: the Brownian motion.  
\end{example}

%%%%%%%%%%
\begin{example}[Time dependent expectation for random walk]
\label{ex:1d-random-time}
Similarly as in Example \ref{ex:1d-random}, we consider the random walk, where we jump left or right each with probability $\half$ on the grid 
\begin{align*}
\{0,\eps,2\eps,\ldots ,1\}\setminus \{0,1\}.
\end{align*}
We also track time: suppose that we are given $$t_0=j \eps^2$$ seconds for some $j\in \N$ at the beginning, and each step reduces $-\eps^2$ seconds of the remaining time. We stop if the token hits the boundary 0 or 1, or at latest when the time runs out (when the remaining time is $0$). The final payoff is given by  $F(x,0)$ and the boundary payoff by $F(1,t)$ or $F(0,t)$. 

What is the expected payoff when starting at $x_0$ with remaining time $t_0$? Let us denote 
\[
\begin{split}
u(x_0, t_0):=\mathbb E^{x_0,t_0}[F(x_{\tau},t_{\tau})],
\end{split}
\]  
where $\tau$ denotes the number of the round when we stop, and $t_{\tau}=t_0-\tau \eps^2$ is the corresponding time.
The expectation satisfies the DPP
\[
\begin{split}
u(x,t)=\half(u(x-\eps,t-\eps^2)+u(x+\eps,t-\eps^2)).
\end{split}
\]
\idea The DPP says that the expectation can be computed by summing up the expectations in the neighboring points with the corresponding step probabilities, but at the time we have left after the step.

We modify the DPP to a form
\begin{align*}
&\frac{u(x,t)-u(x,t-\eps^2)}{\eps^2}\\
&=\frac{1}{2\eps^2}(u(x-\eps,t-\eps^2)-2u(x,t-\eps^2)+u(x+\eps,t-\eps^2)).
\end{align*}
Thus the time dependent random walk seems to be related to the heat equation type problem
\begin{align*}
u_t=\half u_{xx},
\end{align*}
where $u_t$ denotes the partial derivative with respect to $t$ and $u_{xx}$ the second partial derivative with respect to $x$. 
\end{example}

\subsection{Stochastic games and the $p$-Laplacian}
Let $u:\Om\to \R$ with an open bounded domain $\Om \subset \Rn$, and
$$u_{x_1x_1}, u_{x_2 x_2}, \ldots u_{x_n x_n}$$ denote the second partial derivatives to the coordinate directions. The Laplace equation
$$
\Delta u:= u_{x_1x_1}+u_{x_2 x_2}+\ldots +u_{x_n x_n}
$$ 
 is the Euler-Lagrange equation to the Dirichlet integral
\begin{align*}
\int_{\Om} \abs{Du}^2 \ud x.
\end{align*}
Here $Du$ denotes the gradient  of $u$
\[
\begin{split}
Du&
=\begin{pmatrix}
u_{x_1}\\
\vdots\\
u_{x_n}\\ 
\end{pmatrix}
\end{split}
\]
and $\abs{Du}^2=u^2_{x_1}+u_{x_2}^2+\cdots+u_{x_n}^2$.
 A natural nonlinear generalization is the $p$-Dirichlet integral
\begin{align*}
\int_{\Om} \abs{Du}^p \ud x.
\end{align*}
Its Euler-Lagrange equation is the $p$-Laplace equation
\begin{align*}
\operatorname{div}(\abs{Du}^{p-2}Du)=0,
\end{align*}
whose solutions are called $p$-harmonic functions (we will later discuss different concepts of solutions).
Our standing assumption is $$2\le p<\infty$$ for the simplicity of the exposition, even if $1< p\le \infty$ would be  possible as well.
After a computation, assuming $u$ is smooth and $Du\neq 0$, it can be rewritten as 
\begin{align*}
\Delta_p u&:=\operatorname{div}(\abs{Du}^{p-2}Du)\\
&=\abs{Du}^{p-2}\Big(\Delta u+(p-2)\abs{Du}^{-2}\langle D^2 uDu,Du \rangle \Big)\\
&=\abs{Du}^{p-2}\Big(\Delta u+(p-2)\abs{Du}^{-2}\sum_{i,j=1}^{n} u_{x_ix_j} u_{x_i} u_{x_j} \Big)=0,
\end{align*} 
where $u_{x_i}$ denotes the first and $u_{x_ix_j}$ the second partial derivatives.  Above 
\[
\begin{split}
D^2 u&=\begin{pmatrix}
u_{x_1x_1}&\ldots&u_{x_1x_n}\\
\vdots&\ddots&\vdots\\
u_{x_nx_1}&\ldots& u_{x_nx_n}
\end{pmatrix}
\end{split}
\]
is the Hessian matrix of second derivatives.
We often use a shorthand notation 
\begin{align*}
\Delta^{N}_{\infty}u:=\abs{Du}^{-2}\sum_{i,j=1}^{n} u_{x_ix_j} u_{x_i}u_{x_j}=\abs{Du}^{-2} \langle D^2 uDu,Du \rangle,
\end{align*}
for the normalized (also called as game theoretic) infinity Laplacian.  Moreover,
\[
\begin{split}
 \Delta_p^N u: =\Delta u+(p-2)\Delta_\infty^N u
\end{split}
\]
denotes the normalized (also called as game theoretic) $p$-Laplacian whenever $p<\infty$. 

Formally, if we consider the $p$-Laplace equation with the zero right hand side, we can divide/multiply by $\abs{D u}^{p-2}$ as long as the gradient is nonzero. This suggests that $\Delta_p u=0$ and $\Delta_p^N u=0$ are equivalent, which turns out to be the case with suitable interpretations.
 The name infinity Laplacian can be motivated by dividing the $p$-Laplace equation by $p$ and passing formally to the limit $p\to \infty$. For more on the $p$-Laplacian, see \cite{heinonenkm06} and \cite{lindqvist17}.

\subsubsection{Tug-of-war with noise}  
\label{sec:tgw}

It was a remarkable\footnote{From \cite{lindqvist17}: 'I cannot resist mentioning that it was one of the greatest mathematical surprises I have ever had, when I heard about the discovery.'}  discovery of Peres, Schramm, Sheffield and Wilson \cite{peresssw09} that a mathematical game called tug-of-war is connected to the infinity Laplace equation. To be more precise, value functions of the tug-of-war game approximate the infinity harmonic functions as the step size in the game tends to zero. A similar connection holds between the tug-of-war with noise and $p$-harmonic functions \cite{peress08}. 
Our presentation  mainly follows \cite{manfredipr12} in this section. 

%%%%%%%%%%
\noindent{\bf Random walk on balls:}
Let $\Om\subset \Rn$ be a bounded domain. As a first example consider a random walk, where the process is started at $x_0\in \Om$. The next point $$x_1\in B_\eps(x_0)$$ is chosen at random according to the uniform probability distribution on $B_\eps(x_0)$, a ball of radius $\eps$ centered at $x_0$. Then we continue in the same way from $x_1$ until we exit the bounded domain $\Om$. We get a payoff $F(x_{\tau})$, where $x_{\tau}$ is the first point outside $\Om$.

%%%%%%%%%%

 \noindent{\bf Tug-of-war with noise:} Next we add two competing players. As illustrated in Figure~\ref{fig:pelikuva}, a token is placed at $x_0\in \Om\subset \Rn$ and a biased coin with probabilities $\alpha$ and $\beta$ ($\alpha+\beta=1,\ \beta>0, \a\ge 0$) is tossed. If we get heads (probability $\beta$), then the next point $x_1\in B_\eps(x_0)$ is chosen according to the uniform probability distribution on $B_\eps(x_0)$ like above in the random walk on balls. 
 
 On the other hand, if we get tails (probability $\alpha$), then a fair coin is tossed and a winner of the toss moves the token to $x_1\in B_\eps(x_0)$. The players continue playing until they exit the bounded domain $\Om$. At the end of the game, Player II pays Player I the amount given by the payoff function $F(x_\tau)$, where $x_\tau$ is the first point outside $\Om$. 
%%%%%%%%%%%%%%%%%%%%%%

 \begin{figure}[h]
  \begin{tikzpicture}
    \draw[red, very thick] plot [smooth] coordinates {(1,0.78) (0,1.1) (-1,4) (0,6.8) (1,7) (1.5,4.5) (2,4) (2.5,4) (3,5) (3.9,4.4) (4,3) (2,1.1) (1,0.78)};
    \draw[blue, dashed, thick] (1,3) circle (0.5);
    \filldraw[black] (1,3) circle (1pt);
    \draw[thick] (1,3) -- (1.5,3) -- (1.6,3.4) -- (2.1,3.4) -- (2.45,3.75) -- (2.4,4.1);
    \filldraw[black] (2.4,4.1) circle (2pt) node[right] {$F(x_\tau)$};
    \draw[black, dashed, thick] (3.55,5.35) -- (4,5.1) -- (4.3,4.8) -- (4.52,4.35) -- (4.6,4.15) -- (4.65,3.65) -- (4.6,3.15);
    \node[left] at (4.75,3) {$\Gamma_\eps$};
    \node[left] at (1.1,3) {$x_0$};
    \node[below] at (1.25,3.05) {$\eps$};
    %chart
    \filldraw[black] (6,3) circle (1pt);
    \draw[thick, ->] (6,3) -- (8,4);
    \draw[thick, ->] (6,3) -- (8,2);
    \draw[thick, ->] (8,4) -- (9,5);
    \draw[thick, ->] (8,4) -- (9,3);
    \node[right] at (8,2) {a random point,};
    \node[right] at (8,1.6) {"noise"};
    \node[right] at (9,5) {max player};
    \node[right] at (9,3) {min player};
    \node[above] at (7,3.5) {$\alpha$};
    \node[above] at (7,4) {$\alpha+\beta=1$};
    \node[below] at (7,2.5) {$\beta$};
    \node[above] at (8.5,4.5) {I};
    \node[below] at (8.4,3.5) {II};
    \node at (8.8,4.4) {$\frac{1}{2}$};
    \node at (8.8,3.6) {$\frac{1}{2}$};
    \node[above] at (7,6.5) {At the end Player II pays Player I};
    \node[below] at (7,6.5) {the amount $F(x_\tau)$.};
  \end{tikzpicture}
    \caption{Tug-of-war with noise.}
    \label{fig:pelikuva}
 \end{figure}
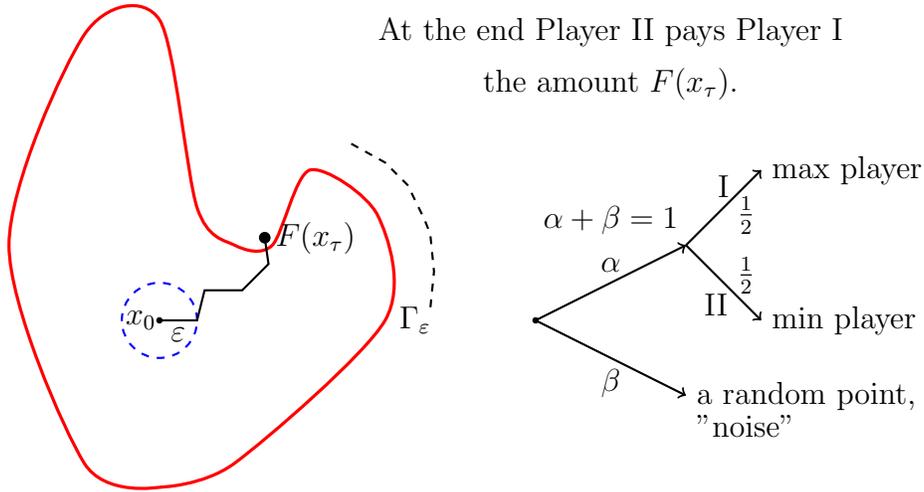

%%%%%%%%%%%%%%%%%%%%%%

We define the value of the game as an expectation over all the possible outcomes $F(x_{\tau})$ when Player I tries to maximize the outcome and Player II tries to minimize it. More precisely, the value\footnote{We later show that the order of $\inf$ and $\sup$ is irrelevant.} of the game is 
\begin{align}
\label{eq:elliptic-value}
u_\eps(x_0):=\sup_{S_{\I}}\inf_{S_{\II}}\,\mathbb{E}_{S_{\I},S_{\II}}^{x_0}[F(x_\tau)],
\end{align}
where  $x_0$ is a starting point of the game and $S_\I,S_{\II}$ are the strategies of the players. The strategies are Borel measurable functions acting on the history of the game and giving the next step in case the player wins the toss. In particular $$S_{\I}(x_0,\ldots,x_k)\in B_{\eps}(x_k).$$ 

Whenever the starting point and the strategies are fixed, the above rules of the game can be interpreted as a one step probability measure. Then the  Kolmogorov-Tulcea construction allows us to construct a probability measure on the space of game sequences, which is needed to define the above expectation. Moreover, since $\beta>0$ there is always some noise. 
 Since $\Om$ is bounded, the random steps force the game to end with probability one\footnote{Since $\beta>0$ and steps are independent, there is a positive probability $\theta$ of stepping further out consecutively until  exiting the domain. Then by the repeated trials, there is an upper bound of the form $(1-\theta)^k\to 0$, as $k\to \infty$, for the event that such consecutive stepping never occurs.}, and the above expectation is well defined. 

As we will later prove, the value satisfies the dynamic programming principle (DPP)
\begin{equation}
\label{eq:DPP-heur}
\begin{split}
u_{\eps}(x) =&\frac{\a}{2}\left\{ \sup_{B_{\eps}(x)} u_{\eps} + \inf_{B_{\eps}(x)}u_{\eps}\right\}+ \beta
\kint_{B_{\eps}(0)} u_{\eps}(x+h) \ud h
\end{split}
\end{equation} 
with 
\[
\begin{split}
u_{\eps}=F\ \text{on}\ \Gamma_\eps:=\{x\in \RR^n \setminus \Om\,:\,\dist(x, \Om) \leq \eps\}.
\end{split}
\]
 Above we used the notation
\[
\kint_{ B_{\eps}(0)} u_{\eps}(x+h) \ud h=\frac{1}{\abs{B_{\eps}(0)}}\int_{ B_{\eps}(0)} u_{\eps}(x+h) \ud h.
\]
 
 \idea Heuristically, in (\ref{eq:DPP-heur}) we get the value at $x$ by considering one step and  by summing up the three possible outcomes (either Player I or Player II wins the toss, or the random step occurs) with the corresponding probabilities.
 
We set
\[
\begin{split}
\a=\frac{p-2}{p+n},\quad \beta=1-\a=\frac{2+n}{p+n},
\end{split}
\]
for $2\le p<\infty$ in order to consider the connection to $p$-harmonic functions. The condition $2\le p$ implies that $\alpha\ge 0$, and  $p<\infty$ implies $\beta>0$. %%%%%%%%%%%%%%%%%%%%
\begin{example}
\label{ex:tgw-p2}
When $p=2$, then $\a=0$ and $\beta=1$. The tug-of-war with noise reduces to the random walk on balls, where the next point at $x$ is chosen according to the uniform probability distribution on $B_{\eps}(x)$ and the DPP takes the form
\[
\begin{split}
u_{\eps}(x)=\kint_{B_{\eps}(0)} u_{\eps}(x+h) \ud h.
\end{split}
\]  
 This can also be compared with Example \ref{ex:1d-random} and \ref{ex:2d-discrete}, where we had a discrete mean in the DPP.
\end{example}

One of the main results in these notes is that the value functions for the tug-of-war with noise converge to the corresponding $p$-harmonic functions (we give nonoptimal assumptions below for simplicity, and postpone the discussion on PDEs, the concept of a solution and uniqueness to Section \ref{sec:PDEs}). This is also illustrated in Figure~\ref{fig:convergence}. 
\begin{theorem}
\label{thm:sol-to-p-Laplace}
Let $\Om$ be a smooth domain, $F:\Rn\setminus \Om \to \R$ a smooth function, and $u$ the unique $p$-harmonic function to the problem
\[
\begin{split}
\begin{cases}
\operatorname{div}(\abs{D u}^{p-2}D u)=0 & \text{ in }\Om, \\
 u=F& \text{ on }\partial \Om. 
\end{cases}
\end{split}
\]
Further, let $u_{\eps}$ be the value function for the tug-of-war with noise with the boundary payoff function $F$. Then 
\[
\begin{split}
u_{\eps} \to u \text{ uniformly on } \ol \Om,
\end{split}
\]
as $\eps\to 0$.
\end{theorem}
\begin{figure}[h]
\includegraphics[scale=0.4]{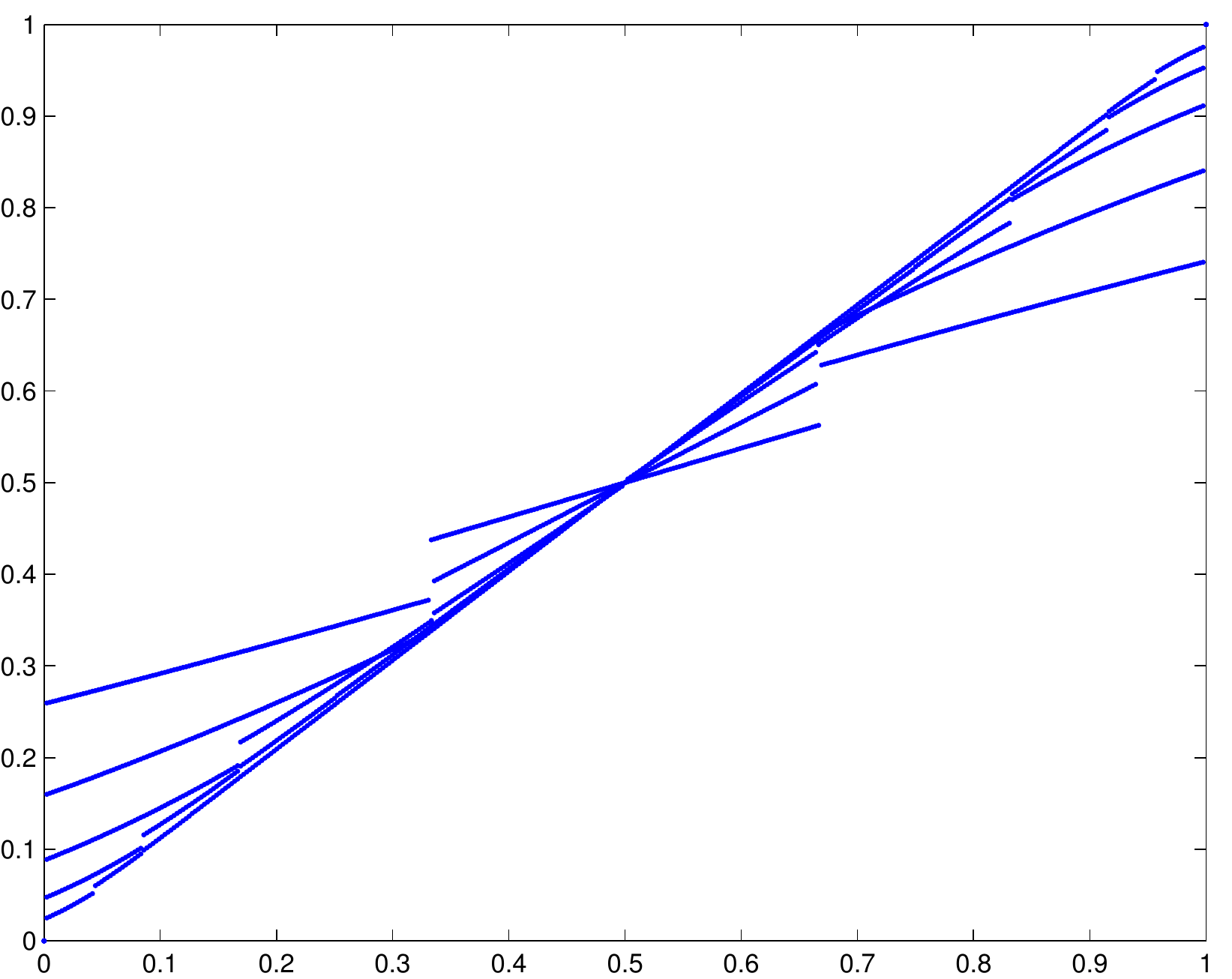}
\caption{Convergence of value functions on $[0,1]$ as the step size $\eps$ gets smaller. One dimensional example with boundary values 0 and 1 is simple, and the limit is just a line for any $p$. In terms of the game, the maximizing player (Player I) pulls to the right, and the minimizing player (Player II) to the left.}
\label{fig:convergence}
\end{figure}
More details on the proof is given in Section~\ref{sec:time-independent}, but here we give a heuristic computation explaining why the theorem should look reasonable. First we take the average of the usual Taylor expansion
\[
u(x+h)=u(x)+\langle D u(x), h \rangle 
+\frac12 \langle D^2u(x)h , h\rangle+\text{Error},
\]
over $B_{\eps}(0)$, where 
\begin{align*}
\langle D^2u\,h , h\rangle=\sum_{i,j=1}^n u_{x_ix_j}h_ih_j.
\end{align*}
We get
\[
\begin{split}
&\kint_{B_{\eps}(0)}u(x+h)\ud h\\
&=u(x)+\kint_{B_{\eps}(0)} \langle D u(x), h \rangle  \ud
h+ \half\kint_{B_{\eps}(0)}\langle D^2u(x)
h,h\rangle \ud h
+\text{Error}.
\end{split}
\]
Because of symmetry, the first integral on the right hand side vanishes. Similarly in the second integral, only the second derivatives to the coordinate directions remain
\begin{align*}
\half&\kint_{B_{\eps}(0)}\langle D^2u\,h,h\rangle \ud h=\half\kint_{B_{\eps}(0)}\sum_{i,j=1}^n  u_{x_ix_j}h_ih_j \ud h\\
&=\half\sum_{i,j=1}^n  u_{x_ix_j}\kint_{B_{\eps}(0)}h_ih_j \ud h
=\half\sum_{i=1}^n  u_{x_ix_i}\kint_{B_{\eps}(0)}h_i^2 \ud h.
\end{align*}
Observe that  by symmetry for any $i$, we can calculate
\begin{align*}
\kint_{B_{\eps}(0)}h_i^2 \ud h&=\frac{1}{n}\kint_{B_{\eps}(0)}\sum_{j=1}^n h_j^{2} \ud h\\
&=\frac{1}{n}\kint_{B_{\eps}(0)} \abs{h}^{2} \ud h\\
&=\frac1{n\abs{B_{\eps}}} \int_0^{\eps}\int_{\partial B_{\rho}(0)}\rho^2 \ud S \ud \rho\, \\ &=\frac1{n\abs{B_{\eps}}} \int_0^{\eps}\sigma_{n-1}\rho^{n-1}\rho^2 \ud \rho\, =\frac{\sigma_{n-1}}{n(n+2)\om_{n}\eps^{n}} \eps^{n+2} \\
&=\frac{\sigma_{n-1}}{n(n+2)\om_{n}} \eps^{2}.
\end{align*}
Above $\om_n$ denotes the measure of the $n$-dimensional unit ball $B_1(0)$ and $\sigma_{n-1}$ the measure of the sphere  $\partial B_1(0)$.
Observing that
 \[
\begin{split}
\om_{n}=\int_0^1 \sigma_{n-1} \rho^{n-1}\ud \rho=\frac{\sigma_{n-1}}{n},
\end{split}
\]
 we get
 \[
\begin{split}
\frac{\sigma_{n-1}}{n(n+2)\om_{n}} \eps^{2} =\frac{\eps^2}{n+2} .
\end{split}
\]
Thus we have
\begin{align*}
\half&\sum_{i=1}^n  u_{x_ix_i}\kint_{B_{\eps}(0)}h_i^2 \ud h=\frac{\eps^2}{2(n+2)} \Delta u.
\end{align*}
Combining the above, we finally obtain
\begin{equation}
\label{exp.lapla} \intav_{B_{\eps}(0)} u(x+h)\ud h =  u(x) +
  \frac{\eps^2}{2(n+2)}\Delta u (x)  + \text{Error}.
\end{equation}

On the other hand, by assuming $D u(x)\neq 0$, and evaluating the Taylor expansion with $h=\pm \eps\frac{D u(x)}{\abs{D
u(x)}}$, we have
\begin{align*}
u\left(x+ \eps\frac{D u}{\abs{D
u}}\right)&\approx u(x)+\left\langle  D u(x),  \eps\frac{D u}{\abs{D
u}} \right\rangle 
+\frac{\eps^2}{2} \left\langle  D^2u(x) \frac{D u}{\abs{D
u}} ,  \frac{D u}{\abs{D
u}}\right\rangle ,\\
u\left(x-\eps\frac{D u}{\abs{D
u}}\right)&\approx u(x)-\left\langle  D u(x),  \eps\frac{D u}{\abs{D
u}} \right\rangle  
+\frac{\eps^2}{2} \left\langle  D^2u(x) \frac{D u}{\abs{D
u}} , \frac{D u}{\abs{D
u}}\right\rangle,
\end{align*}
where we dropped variables and '$+ \text{Error}$' to save space. Further, we approximate\footnote{For details of such an approximation, see for example \cite[Lemma 2.3]{peress08} or \cite[Theorem 3.4]{lewicka20} but our actual proof in Section~\ref{sec:convergence-p}  is organized slightly differently.}
\begin{align*}
u\left(x+ \eps\frac{D u(x)}{\abs{D u(x)}}\right)\approx \sup_{B_{\eps}(x)} u,\quad 
u\left(x-\eps\frac{D u(x)}{\abs{D u(x)}}\right)\approx \inf_{B_{\eps}(x)}u.
\end{align*}
Combining the above approximations, we get
\begin{equation}
\label{exp.infty}
\begin{split}
\frac{1}{2} &\left\{ \sup_{B_{\eps}(x)} u + \inf_{B_{\eps}(x)}u\right\}\\
&=\frac{1}{2} \left\{ u \left(x+\eps\frac{D u(x)}{\abs{D
u(x)}} \right)+ u\left(x-\eps\frac{D u(x)}{\abs{D u(x)}}\right)\right\}+ \text{Error}\\
&= u(x)+ \frac{\eps^2}{2} \Delta^N_\infty u (x) + \text{Error},
\end{split}
\end{equation}
where we recall
\begin{align*}
 \Delta_\infty^N u=\abs{D u}^{-2}\sum_{i,j=1}^{n}  u_{x_ix_j} u_{x_i}u_{x_j}=\left\langle  D^2 u\frac{D u}{\abs{D
u}} , \frac{D u}{\abs{D
u}}\right\rangle .
\end{align*}
Next if we multiply \eqref{exp.lapla} and \eqref{exp.infty} by $\beta$ and $\alpha$ respectively, and add
up the formulas, we get the normalized $p$-Laplace operator on the right hand side i.e.\
\begin{equation}
\label{eq:asymp-exp}
\begin{split}
u(x) =&\frac{\a}{2}\left\{ \sup_{B_{\eps}(x)} u + \inf_{B_{\eps}(x)}u\right\}+ \beta
\kint_{B_{\eps}(0)} u(x+h) \ud h \\
&- \frac{\beta \eps^2}{2(n+2)}\underbrace{(\Delta u+ (p-2)\Delta_{\infty}^N u (x))}_{\Delta_p^N u}+\text{Error}.
\end{split}
\end{equation}
Now, omitting the Error, if a smooth function with nonvanishing gradient satisfies the DPP (\ref{eq:DPP-heur}) at $x$, then  $\Delta_p^N u(x)=0$ and vice versa. This  suggests a connection between game values and $p$-harmonic functions. 

\subsubsection{Time dependent values and tug-of-war with noise}
\label{sec:intro-time-dependent} There is also a time tracking version of the tug-of-war with noise which is related to the normalized (or game theoretic) $p$-parabolic equation
\[
\begin{split}
(n+p)u_t=\Delta_p^N u,
\end{split}
\]
as described in \cite{manfredipr10c}.
The game is otherwise similar to the tug-of-war with noise in the previous section, but we also keep track of time similar to Example \ref{ex:1d-random-time}: suppose that we are given $t_0>0$ seconds at the beginning, and each step reduces $\eps^2/2$ (factor $1/2$ is only for convenience here) seconds of the remaining time. In other words, if we had time $t$ left before the step, after the step we have 
\begin{align*}
t-\frac{\eps^2}{2}.
\end{align*}
The game ends if we exit the domain, or at latest when the time runs out (remaining time is $\le 0$), and $\tau$ denotes the number of the round when we stop. Let $t_{\tau}=t_0-\tau \eps^2/2$ denote the time when we exit the domain or run out of time, and $x_{\tau}$ the corresponding point. The payoff is given by $F(x_{\tau},t_{\tau})$. In the case when the time runs out, $F(\cdot,t_{\tau})$ also needs to be defined inside the domain.

We define the game value in a similar way  as in (\ref{eq:elliptic-value}) by taking $\inf$  over strategies of Player II and $\sup$ over strategies of Player I
\[
\begin{split}
u_{\eps}(x_0,t_0)=\sup_{S_{\I}}\inf_{S_{\II}}\mathbb E^{x_0,t_0}[F(x_\tau,t_\tau)].
\end{split}
\]
As we will later see, the game values satisfy the dynamic programming principle (DPP)
\[
\begin{split}
u_{\eps}(x,t)&= \frac{\alpha}{2}
\left\{ \sup_{y\in  B_\eps (x)} u_{\eps}\Big(y,t-\frac{\eps^2}{2}\Big) +
\inf_{y\in  B_\eps (x)} u_{\eps}\Big(y,t-\frac{\eps^2}{2}\Big) \right\}
\\
&\hspace{1 em}+\beta \kint_{B_\eps (x)} u_{\eps}\Big(y,t-\frac{\eps^2}{2}\Big)
\ud y .
\end{split}
\]
\idea We get the value at the point $x$ at time $t$ by considering one step  and summing up the three possible outcomes that happen at time  $t-\frac{\eps^2}{2}$ (either Player I or Player II wins the toss, or a random step occurs) with the corresponding probabilities.

\begin{example}[Time dependent expectation in the random walk]
\label{eq:rw-time} Let $\alpha=0$ and $\beta=1$.
This means that at $x$ we choose a random point in $B_{\eps}(x)$ according to the uniform probability distribution and the step takes $\eps^2/2$ seconds. 
Now $u$ simply denotes the expected payoff 
\[
\begin{split}
u_{\eps}(x_0,t_0)=\mathbb E^{x_0,t_0}[F(x_\tau,t_\tau)].
\end{split}
\]
The DPP reads as
\begin{align*}
u_{\eps}(x,t)= \kint_{B_\eps (x)} u_{\eps}\Big(y,t-\frac{\eps^2}{2}\Big)
\ud y.
\end{align*}
\end{example}

Next we define a shorthand notation 
\begin{align*}
\Om_T=\Omega\times (0,T),
\end{align*}
and
the parabolic boundary 
\begin{equation}
\label{eq:parabolic-boundary}
\begin{split}
\partial_p \Om_T=\partial_p (\Om\times(0,T))=(\ol \Om\times \{0\})\cup (\partial \Om \times (0,T]).
\end{split}
\end{equation}
Details on viscosity solutions and their uniqueness are postponed to Section~\ref{sec:PDEs}.
\begin{theorem}
\label{thm:sol-to-p-parabolic}
Let $\Om$ be a smooth domain, $F:\R^{n+1}\setminus\Om_T\to \R$ a smooth function, and $u$ the unique viscosity solution to the normalized $p$-parabolic boundary value problem
\[
\begin{split}
\begin{cases}
(n+p)u_t-\Delta_p^N u=0 & \text{ in }\Om \times (0,T]\\
 u=F& \text{ on }\partial_p \Om_T.
\end{cases}
\end{split}
\]
Further, let $u_{\eps}$ be the value function for the time tracking tug-of-war with noise with the boundary payoff function $F$. Then 
\[
\begin{split}
u_{\eps} \to u \text{ uniformly on } \ol \Om\times [0,T],
\end{split}
\]
as $\eps\to 0$.
\end{theorem}

More details on the proof is given in Section~\ref{sec:parab-convergence}, but here we give a heuristic justification.
For simplicity, we concentrate on random walk, where  $p=2$ i.e.\ $\alpha=0$, and $\beta=1$.
By the Taylor expansion
\[
\begin{split}
&\kint_{B_{\eps}(0)}u(x+h,t-s)\ud h\\
&=u(x,t)+\kint_{B_{\eps}(0)} \langle D u(x,t), h \rangle  \ud
h+ \half\kint_{B_{\eps}(0)}\langle D^2u(x,t)
h,h\rangle \ud h\\
&- u_t(x,t)s
+\text{Error}.
\end{split}
\]
where error depends on $\eps^2+s$. If we choose $s=\eps^2/2$, the error depends on $\eps^2$.
Similarly as in (\ref{exp.lapla}), we have
\[
\begin{split}
&\kint_{B_{\eps}(0)}u(x+h,t-s)\ud h\\
&=u(x,t)+0+ \frac{\eps^2}{2(n+2)}\Big(\Delta u(x,t)- (n+2) u_t(x,t)\Big)+\text{Error}.
\end{split}
\]
Thus this suggests a connection to the heat equation type problems.

\section{Viscosity solutions} 
\label{sec:PDEs}

In this section, we give a brief introduction to viscosity solutions. For more on viscosity solutions, see for example \cite{koike04}, \cite{giga06}, \cite{lindqvist12} or \cite{katzourakis15}.

\subsection{Elliptic equations}

We say that a function $\phi$ touches a function $u$ from above at $x\in \Om$, if 
\[
\begin{split}
\phi(x)=u(x)\quad \text{and}\quad  u(y)<\phi(y) \text{ when }y\neq x.
\end{split}
\] 
Similarly, we say that a function $\phi$ touches a function $u$ from below at $x\in \Om$, if 
\[
\begin{split}
\phi(x)=u(x)\quad \text{and}\quad  u(y)>\phi(y) \text{ when }y\neq x.
\end{split}
\] 
Recall the notation
\[
\begin{split}
\Delta_p u&:=\operatorname{div}(\abs{Du}^{p-2}Du).
\end{split}
\]
Also recall that $C(\Om)$ denotes the space of continuous functions in $\Om$, and $C^{2}(\Om)$ the space of twice continuously differentiable functions in $\Om$.
\begin{definition}[Viscosity solution]
\label{eq:viscosity-sol}
A function $u\in C(\Om)$ is a viscosity solution to the $p$-Laplace equation if 
whenever $\phi\in C^{2}(\Om)$ touches $u$ at $x\in \Om$ from  below it holds that
\[
\begin{split}
\Delta_p \phi(x)\le 0,
\end{split}
\] 
and whenever $\phi\in C^{2}(\Om)$ touches $u$ at $x\in \Om$ from  above it holds that
 \[
\begin{split}
\Delta_p \phi(x)\ge 0.
\end{split}
\]
\end{definition}

Here are some further remarks:
\begin{itemize}
\item A function $u$ satisfying the first half 
of the definition  is called a supersolution, and the function satisfying the second half is called a subsolution. It turns out that a supersolution remains above a solution if the boundary values are in the same order, and  similarly a subsolution remains below a solution. 
\item Observe that if $u$ is a smooth solution to 
\begin{align*}
\Delta u=0\text{ in }\Om,
\end{align*}
and  $\phi\in C^{2}(\Om)$ touches $u$ at $x\in \Om$ from  below, then 
\begin{align*}
u-\phi \text{ has a local min at }x.
\end{align*}
For local minimum, it holds
\begin{align*}
D^{2}u(x)-D^{2}\phi(x) \ge  0
\end{align*}
meaning $\left\langle \left(D^{2}u(x)-D^{2}\phi(x)\right)\eta,\eta \right\rangle \ge 0$ for any $\eta\in \R^n,\ \abs{\eta}=1$. From this it follows that 
\begin{align*}
0=\Delta u(x)\ge \Delta \phi(x)
\end{align*}
so a classical solution is a viscosity supersolution (and also subsolution by a similar argument).
\item We required strict touching in the definition above, but we could have used 
\[
\begin{split}
\phi(x)=u(x)\quad \text{and}\quad  u(y)\ge \phi(y) \text{ when }y\neq x,
\end{split}
\] 
as well. This can be seen looking at $\psi(y):=\phi(y)-\abs{y-x}^{4}$ which touches $u$ strictly from below, and for which Definition~\ref{eq:viscosity-sol}  implies $0\ge \Delta_p \psi(x)=\Delta_p \phi(x).$
\item Historically the definition of viscosity solutions was motivated by the vanishing viscosity method:
Solve 
\[
\begin{split}
\begin{cases}
(u')^2=1&\text{ in } (-1,1),\\
u(\pm 1)=0& 
\end{cases}
\end{split}
\]
by adding a smoothing viscosity term $-\eps u_\eps''(x)$. This gives
\[
\begin{split}
\begin{cases}
-\eps u_\eps''(x)+(u_\eps')^2=1 & \text{ in } (-1,1),\\
 u_\eps(\pm1)=0.& 
\end{cases}
\end{split}
\]
By  passing to the limit $\eps\to 0$ (viscosity vanishes), we obtain 
\[
\begin{split}
u_{\eps}(x)\to 1-\abs{x}
\end{split}
\]
uniformly.  One can show that this is the unique viscosity solution to the original equation. Observe that it is not a classical solution!

\item Since $u\in C(\Om)$, one cannot always find $\phi\in C^{2}(\Om)$ touching $u$ from below (or above). If the set of test functions is empty, then the definition is automatically satisfied. 
 \item There are several equivalent variants of the definition. Instead of touching, one can take local $\min$/$\max$ of $u-\phi$, or use second order semi jets $J^{2,\pm} u(x), \ol J^{2,\pm} u(x)$ detailed for example in \cite[Section 2.2]{koike04}. The jets highlight the fact that it is just the ingredients in the second order Taylor expansion that are needed in the definition of viscosity solutions to second order equations.
 \end{itemize}

After these remarks, recall that
\[
\begin{split}
\Delta_p u&=\divt\left(|D u|^{p-2}
D u\right)=|D u|^{p-2}(\Delta u+(p-2)\Delta^N_{\infty}u)\\
&=|D u|^{p-2}\Delta_{p}^N u.
\end{split}
\] 
Observe that in Definition \ref{eq:viscosity-sol}, for $p>2$,
we can neglect the test functions with $D\phi(x)=0$ at the point of touching, since the equation is always satisfied in this case. This   also holds for the full range $1<p<\infty$ by \cite{juutinenlm01} so that the definition still guarantees the uniqueness. Thus we can rewrite the viscosity definition, Definition \ref{eq:viscosity-sol}, as follows.

%%%%%%%%%%%%%%%%%%%%

\begin{definition}
\label{eq:def-normalized-visc}
A function $u\in C(\Om)$ is a viscosity solution to the normalized $p$-Laplace equation $\Delta_p^N u=0$ if 
whenever $\phi\in C^{2}(\Om)$ touches $u$ at $x\in \Om$ with $D\phi(x)\neq0$ from  below it holds that
\[
\begin{split}
\Delta_p^N \phi(x)\le  0,
\end{split}
\] 
and whenever $\phi\in C^{2}(\Om)$ touches $u$ at $x\in \Om$ with $D\phi(x)\neq0$ from  above it holds that
 \[
\begin{split}
\Delta_p^N \phi(x)\ge 0.
\end{split}
\]
\end{definition} 

\sloppy
We could follow the usual path of defining viscosity solutions for singular equations through the semicontinuous envelopes of $\Delta_p^N$ like for example in \cite{giga06, crandallil92}. That this is equivalent with the above definition is shown in \cite[Theorem 6]{kawohlmp12}.

It would also be possible to consider viscosity solutions to $\Delta_p^N u=f$ but then the test functions with zero gradient cannot be omitted. As a further remark, if $f>0$ or $f<0$, then the uniqueness follows from the comparison principle \cite[Theorem 5]{kawohlmp12} but if $f$ touches zero or changes sign, to the best of my knowledge, this is an open problem. It is also known that a viscosity solution to $\Delta_p^N u=f$ is a weak (distributional) solution to $\Delta_p u=\abs{Du}^{p-2}\Delta_p^N u=\abs{Du}^{p-2}f,$ but not necessarily vice versa, and $u\in C^{1,\gamma}$, \cite{attouchipr17}.

The following theorem can be found in \cite{juutinenlm01}, which also accounts for the uniqueness in Theorem~\ref{thm:sol-to-p-Laplace}. As stated above, Definitions~\ref{eq:viscosity-sol} and \ref{eq:def-normalized-visc} are equivalent, and thus either of them can be used as a definition of a viscosity solution to the $p$-Laplace equation. 
\begin{theorem}
\label{thm:uniqueness}
Let  $g\in C(\partial \Om)$ and $u\in C(\ol \Om)$ a viscosity solution to the $p$-Laplace equation such that $u=g$ on $\partial \Om$. Then $u$ is unique.
\end{theorem}
Finally, let us remark that by \cite{juutinenlm01}, see also \cite{julinj12}, the weak (distributional) solutions to the $p$-Laplace equation are equivalent to the viscosity solutions to  the $p$-Laplace equation.

\subsection{Parabolic equations}

%More details also in the lecture notes \url{http://users.jyu.fi/~miparvia/Opetus/Viscosity/lecturenoteVisc2015.pdf}. 
Next we consider viscosity solutions to the normalized or game theoretic $p$-parabolic equation 
\begin{align*}
(n+p)u_t=\Delta_p^N u
\end{align*}
with $2\le p<\infty$.
Now the test functions with the vanishing gradient cannot be completely neglected. We take the standard path of defining viscosity solutions for a singular equation through the semicontinuous envelopes  like for example in \cite{giga06, crandallil92}. In the definition of a parabolic viscosity supersolution (the first half below), if the gradient vanishes, the normalized $p$-Laplacian is defined  as
\begin{align*}
\liminf_{0\neq \eta\to 0}& F(\eta, D^2 \phi(x,t))\\
&=\liminf_{0\neq \eta\to 0} \left(\Delta \phi(x,t)+(p-2)   \left\langle D^{2}\phi(x,t)\frac{\eta}{\abs{\eta}},\frac{\eta}{\abs{\eta}} \right\rangle\right)\\
&=\Delta \phi(x,t)+(p-2)\lambda_{\min}(D^{2}\phi(x,t)),  
\end{align*}
where $\lambda_{\min}(D^{2}\phi(x,t))$ is the smallest eigenvalue of $D^{2}\phi(x,t)$. Similarly in the definition of a subsolution (the second half below),  if the gradient vanishes the normalized $p$-Laplacian is defined as
\begin{align*}
\limsup_{0\neq \eta\to 0}& F(\eta, D^2 \phi(x,t))\\
&=\limsup_{0\neq \eta\to 0} \left(\Delta \phi(x,t)+(p-2)   \left\langle D^{2}\phi(x,t)\frac{\eta}{\abs{\eta}},\frac{\eta}{\abs{\eta}} \right\rangle\right)\\
&=\Delta \phi(x,t)+(p-2)\lambda_{\max}(D^{2}\phi(x,t)),  
\end{align*}
where $\lambda_{\max}(D^{2}\phi(x,t))$ is the largest eigenvalue of $D^{2}\phi(x,t)$. Touching by test functions is defined in a similar way as in the elliptic case but now with respect to the variables  $(x,t)\in \Om_T=\Om\times (0,T)$:  $\phi$ touches a function $u$ from below at $(x,t)\in \Om_T$, if 
\[
\begin{split}
\phi(x,t)=u(x,t)\quad \text{and}\quad  u(y,s)>\phi(y,s) \text{ when }(y,s)\neq (x,t).
\end{split}
\]

\begin{definition}[Parabolic viscosity solution]
\label{eq:viscosity-sol-par}
A function $u\in C(\Om_T)$ is a viscosity solution to the normalized $p$-parabolic equation 
\begin{align*}
(n+p)u_t=\Delta_p^N u
\end{align*}
if 
whenever $\phi\in C^{2}(\Om_T)$ touches $u$ at $(x,t)\in \Om_T$ from  below it holds that
\[
\begin{split}
\begin{cases}
 (n+p)\phi_t(x,t)\ge \Delta_p^N \phi(x,t),& \text{ if }D\phi(x,t)\neq 0,\\
 (n+p)\phi_t(x,t)\ge (p-2)\lambda_{\min}(D^2\phi(x,t))+\Delta \phi(x,t)&  \text{ if }D\phi(x,t)=0,
\end{cases}
\end{split}
\] 
and whenever $\phi\in C^{2}(\Om_T)$ touches $u$ at $(x,t)\in \Om_T$ from  above it holds that
\[
\begin{split}
\begin{cases}
 (n+p)\phi_t(x,t)\le \Delta_p^N \phi(x,t),& \text{ if }D\phi(x,t)\neq 0,\\
 (n+p)\phi_t(x,t)\le (p-2)\lambda_{\max}(D^2\phi(x,t))+\Delta \phi(x,t)&  \text{ if }D\phi(x,t)=0.
\end{cases}
\end{split}
\] 
\end{definition}

Here are some further remarks:
\begin{itemize}
\item The equation is now in nondivergence form without an immediate divergence form counterpart.
\item There are again several equivalent ways of writing down the definition of a viscosity solution: Instead of touching at $(x,t)$, one can require local $\min/\max$ of $u-\phi$ at $(x,t)$. We could also use the definition in terms of parabolic semi jets:  $(a,p,X)\in {\mathcal{P}}^{2,+} u(x,t)$ if
\[
\begin{split}
u(x+h,t+s)\le& u(x,t)+\langle  p, h \rangle\\&+\half \langle Xh,h \rangle+a\,s+o(\abs{s}+\abs{h}^2). 
\end{split}
\]
This is because we only use the ingredients of the second order Taylor expansion in the definition of a viscosity solution.
The definitions of ${\mathcal{P}}^{2,-}u(x,t)$, $\ol {\mathcal{P}}^{2,-}u(x,t)$ and $\ol {\mathcal{P}}^{2,+}u(x,t)$ can be found for example in \cite[Section 3.2.1]{giga06}.
\item Future does not affect: When touching $u$ from below at $(x,t)$ (or similarly from above), it is enough to require that $\phi(x,t)=u(x,t)$ and $\phi\le u$ in $(\Om\times (0,T))\cap \{ s \,:\,s<t \}$ as long as the comparison holds for the operator, see \cite{juutinen01}.  
\end{itemize}

The normalized $p$-parabolic equation has a uniqueness result that we stated as a part of Theorem~\ref{thm:sol-to-p-parabolic}. 
\begin{theorem}
Let $g\in C(\partial_p \Om_T)$ and $u \in C(\ol \Om_T)$. If $u$ is a viscosity solution to the normalized $p$-parabolic equation such that $u=g$ on $\partial_p \Om$, then $u$ is unique. 
\end{theorem}
%%%%%%%%%%
%%%%%%%%%%
\begin{proof}
Observe that in the parabolic case we may immediately assume that $v$ is a strict supersolution i.e. 
\[
\begin{split}
(n+p)v_t-\Delta_p^N v>0
\end{split}
\]
in the viscosity sense by considering instead $v+\frac{\eta}{T-t}$, and that $v(x,t)\to \infty$ as $t\to T$.
The proof is by contradiction: We assume that
$u$ and $v$ are viscosity solutions with the same boundary values and yet inside the domain
\[
\begin{split}
u(x_0,t_0)-v(x_0,t_0)=\sup(u-v)>0.
\end{split}
\]
Let
\[
\begin{split}
w_j(x,t,y,s)=u(x,t)-\Big(v(y,s)+\frac{j}{4}\abs{x-y}^4+\frac{j}{2}(t-s)^2\Big)
\end{split}
\]
and denote by $(x_j,t_j,y_j,s_j)$ the maximum point of $w_j$ in
$\ol \Om_T\times \ol \Om_T$. Since $(x_0,t_0)$ is a local maximum
for $u-v$, we may assume that
\[
\begin{split}
(x_j,t_j,y_j,s_j)\to (x_0,t_0,x_0,t_0) \quad \textrm{as} \quad j\to \infty
\end{split}
\]
and $(x_j,t_j)\, ,(y_j,s_j)\in \Om_T$ for all large $j$; we only consider such indices.

We consider two cases: either $x_j=y_j$ infinitely often or $x_j\neq y_j$ for all $j$ large enough. First, let  $x_j=y_j$, and denote
\[
\begin{split}
\phi(y,s)=-\frac{j}{4}\abs{x_j-y}^4-\frac{j}{2}(t_j-s)^2.
\end{split}
\]
Then
\[
\begin{split}
v(y,s)-\phi(y,s)
\end{split}
\]
has a local minimum at $(y_j,s_j)$. Since $v$ is a strict supersolution, we have
\[
\begin{split}
(n+p)\phi_t(y_j,s_j)  > (p-2) \lambda_\textrm{min}(D^2\phi(y_j,s_j))+\Delta \phi(y_j,s_j).
\end{split}
\]
Since  $y_j=x_j$ and thus $D^2\phi(y_j,s_j) =0$, we have by the previous
inequality
\begin{equation}
\label{eq:from-above}
\begin{split}
0&<(n+p)j(t_j-s_j).
\end{split}
\end{equation}

Next redefine the notation
\[
\begin{split}
\phi(x,t)=\frac{j}{4}\abs{x-y_j}^4+\frac{j}{2}(t-s_j)^2.
\end{split}
\]
Similarly,
\[
\begin{split}
u(x,t)-\phi(x,t)
\end{split}
\]
has a local maximum at $(x_j,t_j)$, and thus since $D^2\phi(x_j,t_j)=0$, our assumptions imply
\begin{equation}
\label{eq:from-below}
\begin{split}
0&=\Delta \phi(x_j,t_j)+(p-2)\lambda_{\max}(D^2\phi(x_j,t_j))\ge (p+n)\phi_t(x_j,t_j)\\
&=(p+n)j(t_j-s_j).
\end{split}
\end{equation}
Summing up \eqref{eq:from-above} and \eqref{eq:from-below}, we get
\[
\begin{split}
0<(n+p)j(t_j-s_j)-(n+p)j(t_j-s_j)=0,
\end{split}
\]
a contradiction.

Next we consider the case $y_j\neq x_j$. We use the theorem on sums (also called Ishii's lemma or maximum principle for semicontinuous functions, see for example \cite[Lemma 3.6]{koike04}, or \cite[Theorem 3.3.2]{giga06}), for $w_j$
 which implies that there exist symmetric matrices $X_j,Y_j$ such that $X_j-Y_j \le 0$\footnote{If $w(x,y,t,s)=u(x,t)-v(y,s)$ were smooth, then at a maximum point $(x,t,y,s)$ we have $D^2_{xx} u(x,t)-D^2_{yy} v(y,s)\le 0$. Analogously, $X_j-Y_j \le 0$ by theorem on sums without the smoothness assumption.} and
\[
\begin{split}
&\Big(j(t_j-s_j),\,j\abs{x_j-y_j}^2(x_j-y_j),\,X_j\Big)\in  \ol{\mathcal{P}}^{2,+}u(x_j,t_j)\\
&\Big(j(t_j-s_j),\,j\abs{x_j-y_j}^2(x_j-y_j),\,Y_j\Big)\in \ol{\mathcal{P}}^{2,-}v(y_j,s_j).
\end{split}
\]
Since $v$ is a strict supersolution and $u$ a (sub)solution, we get by alternative definitions in terms of jets 
\[
\begin{split}
0&< (n+p)j(t_j-s_j)- (p-2)\left\langle Y_j \frac{(x_j-y_j)}{\abs{x_j-y_j}},\, \frac{(x_j-y_j)}{\abs{x_j-y_j}}\right\rangle-\tr(Y_j) \\
&\hspace{1 em}-(n+p)j(t_j-s_j)+(p-2)\left\langle X_j  \frac{(x_j-y_j)}{\abs{x_j-y_j}},\,\frac{(x_j-y_j)}{\abs{x_j-y_j}}\right\rangle+\tr(X_j)\\
&\le (p-2)\left\langle (X_j-Y_j)  \frac{(x_j-y_j)}{\abs{x_j-y_j}},\,\frac{(x_j-y_j)}{\abs{x_j-y_j}}\right\rangle+\tr(X_j-Y_j)\le 0,
\end{split}
\]
where the last inequality follows since $X_j-Y_j$ is negative semidefinite. This provides the desired contradiction.
\end{proof}
%%%%%%%%%%

As we already stated, in the elliptic case, we could omit the test functions with $D\phi(x)=0$. In the parabolic case, we cannot omit them but the test function reduction is still possible using the similar technique as above in the uniqueness proof.  For the proof, see \cite[Lemma 2.2]{manfredipr10c}. 
\begin{theorem}
\label{thm:test-functions-neglect}
In the case $D\phi(x,t)=0$, we can further restrict the class of test functions to $D\phi(x,t)=0,D^2\phi(x,t)=0$. 
In other words, the above definition can be rewritten as
\[
\begin{split}
\begin{cases}
 (n+p)\phi_t(x,t)-\Delta_p^N \phi(x,t) \ge 0,& \text{ if }D\phi(x,t)\neq 0,\\
 (n+p)\phi_t(x,t)\ge 0&  \text{ if }D\phi(x,t)=0 \text{ and } D^2\phi(x,t)=0
\end{cases}
\end{split}
\] 
when touching from below, and analogously from above.
\end{theorem}

Both in the elliptic and parabolic case the existence for the Dirichlet boundary value problem can be proven through the Perron method or approximation procedure. This requires suitable regularity assumptions on the boundary. 
Interestingly, the games give an alternative proof for the existence since we will later see that game value functions converge to the viscosity solution. 

For more on the viscosity theory of parabolic equations, see for example \cite{chengg91, giga06, juutinenk06} and \cite{imberts13}. In \cite{does11}, in addition to other results,  the game theoretic $p$-parabolic equation was applied to image processing. 
The game theoretic interpretation has partly inspired  attention on the game theoretic $p$-parabolic equation from the PDE perspective. Banerjee and Garofalo studied the equation from the potential theoretic point of view for example in \cite{banerjeeg13, banerjeeg15}. Jin and Silvestre \cite{jins17} established $C^{1,\gamma}$-regularity. Later, H{\o}eg and Lindqvist \cite{hoegl20}, as well as  Dong, Peng, Zhang and Zhou \cite{dongpzz20} studied higher order Sobolev regularity for the game theoretic $p$-parabolic equation. The theory has partly been extended to a more general class of equations that contains both the game theoretic and standard $p$-parabolic equations for example in \cite{ohnumas97,imbertjs19} and \cite{parviainenv20}.

\section{Stochastic tug-of-war games}

\subsection{Time independent values and the $p$-Laplace equation}
\label{sec:time-independent}

Next we start looking at the connection between stochastic processes and PDEs more closely. Examples \ref{ex:1d-random} and \ref{ex:tgw-p2} in the introduction suggested that there is a connection between a random walk  and harmonic functions.  Here we consider a random walk where when at $x$, the next point is chosen according to the uniform probability distribution on $B_{\eps}(x)\subset \Om\subset \Rn$ and the value is simply the expectation
\begin{align*}
u_{\eps}(x_0):=\mathbb E^{x_0}[F(x_{\tau})],
\end{align*}
where $F$ is a given boundary payoff, $x_{\tau}$ is the first point outside the game domain $\Om\subset \Rn$, $\tau$ the stopping time giving the corresponding round, and $x_0\in \Om$ is a given starting point. 

For simplicity, take $\Om=B_1(0)$ and let the boundary payoff be
$$
F(x)=x_1 \text{ for } x=(x_1,x_2,\ldots,x_n).
$$ 
The harmonic function  with  $\Delta u=0$ and boundary values $F$ is  
$$
u(x)=x_1.
$$
This as well as harmonic functions in general satisfy the mean value principle 
$$
u(x)=\kint_{B_{\eps}(0)} u(x+h) \ud h.
$$

Let us denote by $\E^{x_0}[u(x_k)| x_0,\ldots,x_{k-1}]$ the conditional expectation of $u(x_k)$ knowing the previous positions $x_0,\ldots,x_{k-1}$.
Now $$M_j:=u(x_j)$$ is a martingale under the random walk since 
\begin{align*}
\E^{x_0}[u(x_k)| x_0,\ldots,x_{k-1}]=\kint_{B_{\eps}(0)} u(x_{k-1}+h) \ud h=u(x_{k-1}),
\end{align*}
and this is exactly as in the definition of a martingale. 
By the  optional stopping theorem $\E[M_0]=\E[M_\tau]$ (see for example \cite[Theorem A.31]{lewicka20}) and thus
\begin{align*}
u(x_0)&=\E^{x_0}[u(x_0)]=\E^{x_0}[M_0]=\E^{x_0}[M_\tau]=\E^{x_0}[u(x_\tau)]=\E^{x_0}[F(x_{\tau})]\\
&=u_{\eps}(x_0).
\end{align*}
In other words, the expectation of the random walk coincides with the corresponding harmonic function.

Similarly, if $\Delta u\ge 0$ i.e.\ $u$ is a subsolution, then
$M_j:=u(x_j)$ is a submartingale  under the random walk since 
\begin{align*}
\E^{x_0}[u(x_k)|x_0,\ldots,x_{k-1}]=\kint_{B_{\eps}(0)} u(x_{k-1}+h) \ud h\ge u(x_{k-1}).
\end{align*}

Finally, if $\Delta u\le 0$ i.e.\ $u$ is a supersolution, then
$M_j:=u(x_j)$ is a supermartingale under the random walk since 
\begin{align*}
\E^{x_0}[u(x_k)|x_0,\ldots,x_{k-1}]=\kint_{B_{\eps}(0)} u(x_{k-1}+h) \ud h\le u(x_{k-1}).
\end{align*}

\idea A martingale is expected to remain the same over  one step:
\begin{align*}
\E^{x_0}[u(x_k)| x_0,\ldots,x_{k-1}]=u(x_{k-1}).
\end{align*}
A submartingale is expected to increase over  one step:
\begin{align*}
\E^{x_0}[u(x_k)| x_0,\ldots,x_{k-1}]\ge u(x_{k-1}).
\end{align*}
A supermartingale is expected to decrease over one step:
\begin{align*}
\E^{x_0}[u(x_k)| x_0,\ldots,x_{k-1}]\le u(x_{k-1}).
\end{align*}
By the  optional stopping theorem, we can iterate the known one step behavior of a super-/sub-/martingale over several rounds and even get a stopping time there under suitable assumptions.

 \subsubsection{Dynamic programming principle}

Earlier we conjectured that $p$-harmonic functions are related to the value functions of the tug-of-war with noise and to solutions
of the dynamic programming principle (DPP)
\begin{equation}
\label{eq:DPP-repeat}
\begin{split}
u_{\eps}(x) =&\frac{\a}{2}\left\{ \sup_{B_{\eps}(x)} u_{\eps} + \inf_{B_{\eps}(x)}u_{\eps}\right\}+ \beta
\kint_{B_{\eps}(0)} u_{\eps}(x+h) \ud h
\end{split}
\end{equation} 
with 
\begin{align*}
u_{\eps}=F\ \text{on}\ \Gamma_\eps:=\{x\in \R^n \setminus \Om\,:\,\dist(x,\partial \Om) \leq \eps\},
\end{align*}
and
\[
\begin{split}
\a=\frac{p-2}{p+n},\quad \beta=1-\a=\frac{2+n}{p+n}.
\end{split}
\]
Actually, the functions satisfying the DPP can be taken as a starting point. Such functions are sometimes called $p$-harmonious (not harmonic) functions,  adapting terminology of harmonious extensions \cite{legruyera98} in the case $p=\infty$.  In the classical case $p=2$, related extensions were considered for example by Courant, Friedrichs and Lewy, see also \cite{doyles84}. 

We are going to adopt an approach where we start from the DPP, as it allows us to work out a rather elementary analytic proof. The approach is from \cite{luirops14}.

\idea\begin{itemize}
\item First we are going to show that there exists  a function satisfying the DPP with given boundary values, that is we show the existence of a $p$-harmonious function.
\item Then we will use this $p$-harmonious function to choose 'good' strategies, and show that both players can guarantee the amount given by the $p$-harmonious function. In other words it is the unique value of the game.
\end{itemize}

For simplicity we assume $0<\alpha, \beta<1$ since the endpoints $0$ or $1$ require a different treatment. This corresponds to $2<p<\infty$.
\begin{theorem}
\label{th:dpp}
Given a bounded Borel boundary function
$F:\Gamma_\eps\to \R$, there is a bounded Borel function $u$ that satisfies
the DPP with the boundary values $F$.
\end{theorem}
%%%%%%%%%%%%%%%%%%%%%%%%%%%
\begin{proof}
As illustrated in Figure~\ref{fig:increasing}, the proof is based on the monotone iteration giving a fixed point for the operator 
\[
\begin{split}
Tv(x):=\begin{cases}
 \frac{\a}{2}\left\{ \sup_{B_{\eps}(x)} v + \inf_{B_{\eps}(x)}v\right\}+ \beta
\kint_{B_{\eps}(0)} v(x+h) \ud h, & x\in \Om\\
 F(x),&x\in \Gamma_\eps. 
\end{cases}
\end{split}
\]
\begin{figure}[h]
\includegraphics[scale=0.4, ,trim = 2em 23ex 2em 23ex]{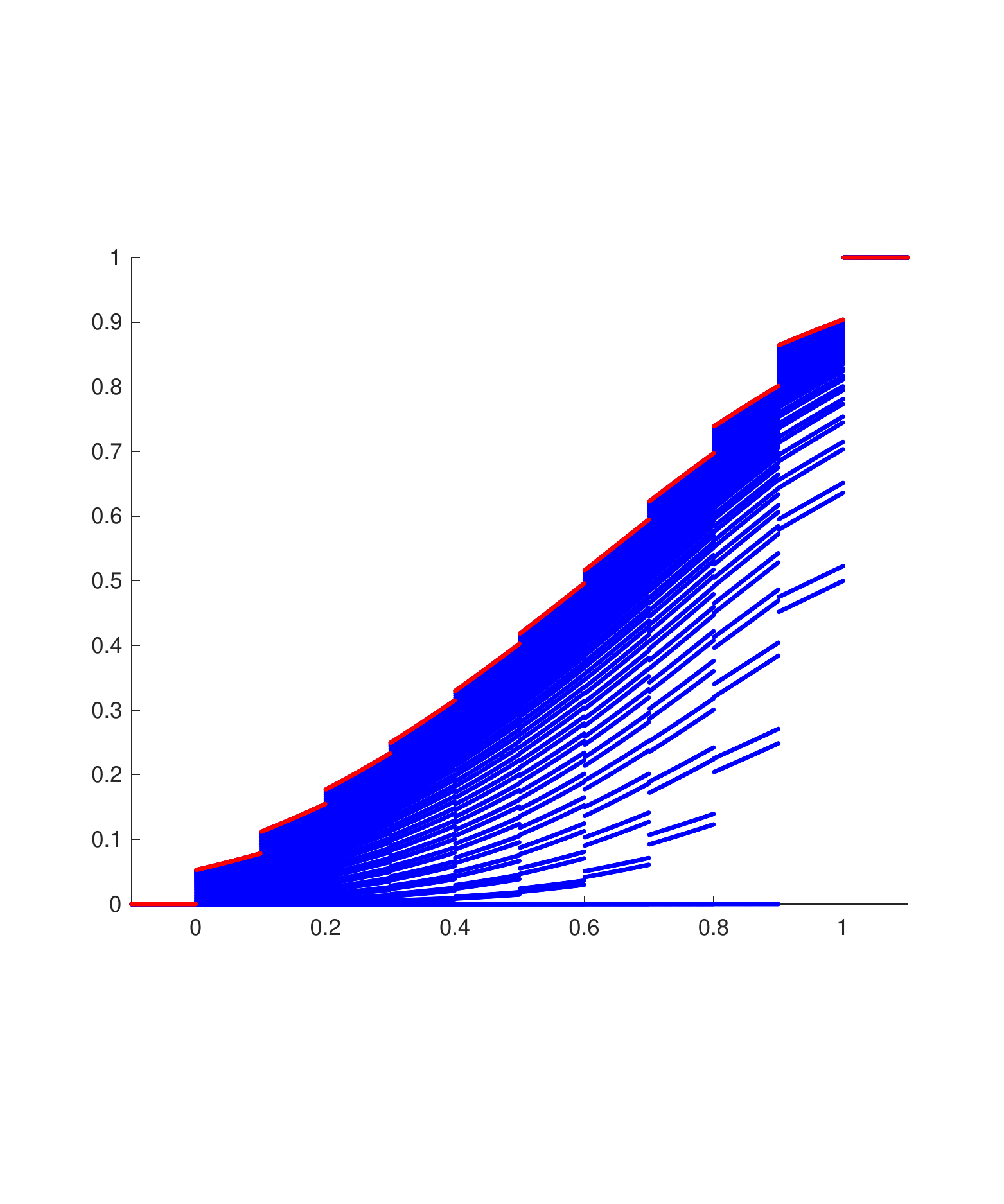}
\caption{One dimensional illustration of the increasing sequence of functions generated in the proof of Theorem \ref{th:dpp}.}
\label{fig:increasing}
\end{figure}

\noindent {\bf Step 1 (pointwise limit):}  We look at the limit
$$
u=\lim_{j\to\infty} u_j, \quad {\rm with}\;\; u_{j+1}=Tu_j\;\;{\rm for}\;\; j=0,1,\ldots,
$$
where 
$$
u_0(x)=\begin{cases} \inf_{y\in\Gamma_\eps} F,& x\in \Omega \\
F(x), &x\in \Gamma_{\eps}.
\end{cases}
$$
Then $u_1\geq u_0$ and further since $T$ preserves the order of functions 
$$
u_2=Tu_1\geq Tu_0=u_1.
$$ 
Similarly, 
\[
\begin{split}
u_{j+1}\geq u_j
\end{split}
\]
 for all $j=0,1,\ldots$. Hence the sequence $u_j$ is increasing and since it is bounded from above by $\sup_{y\in \Gamma_\eps}F(y)<\infty$, we may define the bounded  Borel function $u$ as the monotone pointwise limit
$$
u(x):=\lim_{j\to\infty}u_j(x).
$$
\noindent {\bf Step 2 (uniform convergence):} Pointwise convergence does not guarantee that the limit will satisfy the DPP, and therefore we will show that the convergence is uniform. Suppose contrary to our claim that
\begin{equation}\label{eqM1}
M:=\lim_{j\to\infty}\sup_{x\in\Omega}(u-u_j)(x)>0.
\end{equation}
Fix arbitrary $\delta >0$ and select $k\geq 1$ large enough so that
\begin{equation}\nonumber%\label{eqM2}
u-u_k\leq M+\delta\quad {\rm in}\;\; \Omega.
\end{equation}
By \eqref{eqM1} we may choose $x_0\in \Omega$ with the property 
 $$u(x_0)-u_{k+1}(x_0)\geq M-\delta.
 $$ Then choose $\ell>k$ large enough so that $u(x_0)-u_{\ell+1}(x_0)<\delta$, whence it follows that
\begin{equation}\nonumber%\label{eqM4}
u_{\ell+1}(x_0)-u_{k+1}(x_0)\geq M-2\delta.
\end{equation}
By the dominated convergence theorem,  we may also assume that 
\begin{equation}\nonumber%\label{eqM3}
\sup_{x\in\Omega}\beta\intav_{B_\eps(x)}(u-u_k)(y)\,dy\,\,\leq\delta\,.
\end{equation}

It also holds for any set $A$ that
\begin{equation}
\label{eq:supremums}
\begin{split}
\sup_{A}u_{\ell}-\sup_{A}u_k&\leq \sup_{A}(u_\ell-u_k),\\
\inf_{A}u_{\ell}-\inf_{A}u_k&\leq \sup_{A}(u_\ell-u_k).
\end{split}
\end{equation}
The iterative definition and monotonicity of the sequence $u_j$ together with the above estimates  imply
\begin{eqnarray*}
M-2\delta &\leq &u_{\ell+1}(x_0)-u_{k+1}(x_0)\\
&=&\frac{\alpha}{2}\sup_{B_\eps(x_0)}u_{\ell}+\frac{\alpha}{2}\inf_{B_\eps(x_0)}u_{\ell}+\beta\intav_{B_\eps(x_0)}u_\ell\\
&&\phantom{kkkkkkk}\,\,-\bigg(\frac{\alpha}{2}\sup_{B_\eps(x_0)}u_{k}+\frac{\alpha}{2}\inf_{B_\eps(x_0)}u_{k}+\beta\intav_{B_\eps(x_0)}u_k\bigg)\\
&\leq & \alpha\sup_{B_\eps(x_0)}(u_{\ell}-u_{k})+\beta\intav_{B_\eps(x_0)}(u_\ell-u_k)\\
&\leq & \alpha\sup_{B_\eps(x_0)}(u-u_{k})+\beta\intav_{B_\eps(x_0)}(u-u_k)\\
&\leq &\alpha (M+\delta)+\delta.
\end{eqnarray*}
Since $\alpha<1$, this provides the contradiction if $\delta$ is chosen small enough.

By the uniform convergence, the limit $u$ obviously satisfies the DPP and it has the right boundary values by construction.
\end{proof}

The choice of open balls in the definition of the game and in the DPP is important: it guarantees by writing for example the level set 
$\{ x \,:\, \sup_{B_{\eps}(x)} v>\lambda\}$ in the form 
$
\cup_{y: v(y)>\lambda} B_{\eps}(y)
$
that the functions in the iteration remain Borel.
 One can give an example, \cite[Example 2.4]{luirops14}, of a bounded Borel function such that the function 
 \begin{align*}
x\mapsto 
\sup_{y\in \overline{B}_1(x)}u(y),
\end{align*} where $\overline{B}_1(x)$ is a closed ball, is not Borel. Such a simple operation does not necessarily preserve measurability!

The uniqueness of the $p$-harmonious function with a given boundary data can also be shown by a direct analytic proof \cite[Theorem 2.2]{luirops14}, but we omit the details since this follows later from the fact that the game has a value.

\subsubsection{Existence of a value function}
\label{sec:elliptic-value}

We next verify that the solution to the DPP obtained in the previous section indeed equals the value function of the game. The rules of the tug-of-war with noise can be recalled from Section \ref{sec:tgw}. 

Value of the game for Player I is given by
\[
u^\eps_\I(x_0)=\sup_{S_{\I}}\inf_{S_{\II}}\,\mathbb{E}_{S_{\I},S_{\II}}^{x_0}[F(x_\tau)]
\]
while the value of the game for Player II is given by
\[
u^\eps_\II(x_0)=\inf_{S_{\II}}\sup_{S_{\I}}\,\mathbb{E}_{S_{\I},S_{\II}}^{x_0}[F(x_\tau)],
\]
where $x_0$ is a starting point of the game and $S_\I,S_{\II}$ are the strategies employed by the players. 

 \idea In the value for Player I, Player I takes the supremum over all her strategies of the infimum taken over all Player II's strategies of the expected payoff. In other words, Player I will have to choose first and Player II chooses his strategy only after Player I's strategy is already fixed. In the value for Player II it is the opposite. Thus it always holds that 
\begin{align*}
u^\eps_\I\le u^\eps_\II,
\end{align*}
also cf.\ for example \cite[Lemma 2.6.3]{karlinp17}.
If $u^\eps_\I= u^\eps_\II$, which will be the case here,  it is said that the game has a value. The following proof is from \cite[Theorem 3.2]{luirops13}.
\begin{theorem}[Value of the game]\label{th:value}
It holds that $u=u^\eps_\I = u^\eps_\II$, where $u$ is a solution of the DPP
as in Theorem \ref{th:dpp}. In other words, the game has a unique value. 
\end{theorem}
\begin{proof}
It is enough to show that 
\begin{equation}\label{eq:II} 
u_\II^\eps \leq u
\end{equation}
because by symmetry $u_\I^\eps \geq u$ and always
$u_\II^\eps\geq   u_\I^\eps$.

To prove \eqref{eq:II} we play the game as follows:
Player II follows a strategy
$S_\II^0$ such that at $x_{k-1}\in \Om$ he chooses to step to a
point that almost minimizes $u$, that is, to a point $x_k \in 
B_\eps(x_{k-1})$ such that
\[
u(x_k)\leq\inf_{B_\eps(x_{k-1})} u+\eta 2^{-k}
\]
for some fixed $\eta>0$. Moreover, this strategy can be chosen to be Borel by using Lusin's countable selection theorem even if we  omit the details here.

Then for any strategy $S_\I$, we  use the  definition of the game, definition of $S_\II^0$,  and the fact that $u$ satisfies the DPP to estimate
\[
\begin{split}
\mathbb{E}_{S_\I, S^0_\II}^{x_0}&[u(x_k)+\eta 2^{-k}\,|x_0,\ldots,x_{k-1}]
\\
&=\frac{\alpha}{2} \left\{ u \big(S^0_\II(x_0,\ldots ,x_{k-1})\big)+ u\big(S_\I(x_0,\ldots ,x_{k-1})\big)\right\}\\
&\phantom{aaaaaaaaaaaaa}+ \beta\kint_{ B_\eps(x_{k-1})}
u \ud y+\eta 2^{-k}\\
&\leq \frac{\alpha}{2} \big\{\inf_{ B_\eps(x_{k-1})} u+\eta 2^{-k}+\sup_{  B_\eps(x_{k-1})}
u\big\}+ \beta\kint_{ B_\eps(x_{k-1})}
u \ud y+\eta 2^{-k}\\
&=u(x_{k-1})+\eta 2^{-k}(1+\frac{\alpha}{2})\\
&\leq u(x_{k-1})+\eta 2^{-(k-1)}.
\end{split}
\]
Above $\mathbb{E}_{S_\I, S^0_\II}^{x_0}[\ldots |x_0,\ldots,x_{k-1}]$ denotes the conditional expectation conditioned on the history of the game, and we estimated the outcome of the action of Player I by $\sup$. 
%Here I would not omit eta since one needs to make it smaller every turn.

Thus, regardless of the strategy $S_\I$ the process $
M_k=u(x_k)+\eta 2^{-k}
$ is a supermartingale with respect to the history of the game. Recall that $F(x_\tau)=u(x_\tau)$ where $\tau$ is the usual exit time and observe that the supermartingale is bounded. 
We deduce by the optional stopping theorem that
\[
\begin{split}
u_\II^\eps(x_0)&= \inf_{S_{\II}}\sup_{S_{\I}}\,\mathbb{E}_{S_{\I},S_{\II}}^{x_0}[F(x_\tau)]\le \sup_{S_\I} \mathbb{E}_{S_\I,
S^0_\II}^{x_0}[F(x_\tau)+\eta 2^{-\tau}]\\ 
&\leq\sup_{S_\I}  \mathbb{E}^{x_0}_{S_\I, S^0_\II}[u (x_0)+\eta ]=u(x_0)+\eta.
\end{split}
\]
Since $\eta$ was arbitrary this proves the claim.
\end{proof}

The previous result justifies the notation $u_{\eps}:=u^\eps_\I=u^\eps_\II$.

\subsubsection{Convergence of value functions to the $p$-harmonic function}
\label{sec:convergence-p} 

We stated earlier in Theorem \ref{thm:sol-to-p-Laplace}  that value functions to the tug-of-war with noise converge to the $p$-harmonic function as the step size $\eps\to 0$. We repeat the formulation taking into account what we discussed on viscosity solutions earlier. 
\begin{theorem}
\label{thm:sol-viscosity}
Let $\Om$ be a smooth domain, $F:\Rn\setminus \Om \to \R$ a smooth function, and $u$ the unique viscosity solution to the problem
\[
\begin{split}
\begin{cases}
\Delta_p^N u=0 & \text{ in }\Om, \\
 u=F& \text{ on }\partial \Om. 
\end{cases}
\end{split}
\]
Further, let $u_{\eps}$ be the value function for the tug-of-war with noise with the boundary payoff function $F$. Then 
\[
\begin{split}
u_{\eps} \to u \text{ uniformly on } \ol \Om,
\end{split}
\]
as $\eps\to 0$.
\end{theorem}

\noindent {\sc Steps:}
\begin{itemize}
\item We need a local {\bf regularity} estimate for $u_{\eps}$, which is independent of $\eps$ in a suitable sense, and uniform {\bf boundedness} for $u_{\eps}$. 
\item  We need a regularity estimate at the vicinity of the boundary. 
\item By the first steps, the conditions of the  {\bf Arzel\`a-Ascoli} theorem (or actually a version that allows small jumps \cite[Lemma 4.2]{manfredipr12}) are fulfilled, and we may pass to $\eps\to 0$ (possibly passing to a subsequence) to obtain a continuous limit $u$ such that 
\[
\begin{split}
u_{\eps}\to u
\end{split}
\]
uniformly in $\ol \Om$  as $\eps\to 0$.
\item Then using the DPP, we establish that the limit $u$ is a viscosity solution to $\Delta_p^N u=0$. This solution is also unique as stated in Theorem~\ref{thm:uniqueness}, and thus the limit is independent of the chosen subsequence.
\end{itemize}
To see that the conditions of the Arzel\`a-Ascoli theorem hold, we need the following:
\begin{itemize}
\item Boundedness: Since the value was defined as an expectation over the boundary payoff, which was bounded, also the value functions are uniformly bounded.
\item Regularity at the vicinity of the boundary: we may pull towards a fixed boundary point. Recalling smoothness assumptions,  and using a boundary barrier (cf.\ boundary barriers in the theory of PDEs) given by the underlying PDE, one can deduce the desired regularity estimate. Further details are omitted, but can be found from \cite{peress08,manfredipr12}.
\item Full regularity: One can use 
\begin{itemize}
\item Global approach: copying of strategies and translation invariance allows to 'follow' another game instance to the vicinity of the boundary \cite[Theorem 1.2(i)]{peress08}. Alternative version of the proof uses a comparison principle \cite[Lemma 4.4]{manfredipr12}.
\item Local approach: use cancellation strategies to prove local Lipschitz regularity as in \cite{luirops13} (Section \ref{sec:Lip-regularity}).
\item Further local approaches: couplings of DPPs \cite{luirop18} (Section \ref{sec:other-methods}), or Krylov-Safonov type methods \cite{arroyobp,arroyobp2} (Section \ref{sec:other-methods}).
\end{itemize}
\end{itemize}

There is also an alternative approach \cite{barless91} directly using the comparison principle for the limit equation (if available) that we do not consider here.

\noindent {\bf Verifying that the limit is a viscosity solution:} We postpone the regularity considerations, and instead assume that we have obtained a continuous limit $u$ as sketched above. We now establish that it is a viscosity solution to $\Delta_p^N u=0$ (last step above). This is akin to the stability argument for viscosity solutions under uniform convergence. 

Choose a point $x\in \Omega$ and a $C^2$-function  $\phi$ defined
in a neighborhood of $x$. Let 
\[
\phi(x_1^\varepsilon)=\min_{y\in
\ol B_{\eps}(x)}\phi(y). 
\]
Evaluating the Taylor expansion of
$\phi$ at $x$ with $x_1^\eps$, we get
\[
\phi(x_1^{\varepsilon})=\phi(x) + \langle D\phi(x), (x_1^{\varepsilon}
- x) \rangle  +\frac12 \langle
D^2\phi(x)(x_1^{\varepsilon}-x),(x_1^{\varepsilon}-x)\rangle+o(\varepsilon^2),
\]
where $o(\eps^2)/\eps^2\to 0$ as $\eps\to 0$.
Moreover, using the opposite point $\tilde{x}_1^{\varepsilon}=x +x-x_1^\eps$ for which $\tilde{x}_1^{\varepsilon}-x=-(x_1^\eps-x)$, we also get
\[
\phi(\tilde{x}_1^{\varepsilon})=\phi(x) - \langle D\phi(x)
, (x_1^{\varepsilon} - x) \rangle  +\frac12 \langle
D^2\phi(x)(x_1^{\varepsilon}-x),(x_1^{\varepsilon}-x)\rangle+o(\varepsilon^2).
\]
Adding the expressions, we obtain
\[
\phi(\tilde{x}_1^{\varepsilon}) + \phi(x_1^{\varepsilon})- 2
\phi(x) = \langle
D^2\phi(x)(x_1^{\varepsilon}-x),(x_1^{\varepsilon}-x)\rangle+o(\varepsilon^2).
\]
Since $x_1^{\varepsilon}$ is the point where the minimum of $\phi$ is
attained, it follows that
\[
\phi(\tilde{x}_1^{\varepsilon}) + \phi(x_1^{\varepsilon})- 2
\phi(x) \le \max_{y\in \ol B_{\eps}(x)} \phi (y) +
\min_{y\in \ol B_{\eps}(x)} \phi (y) - 2 \phi(x),
\]
and thus combining the previous observations
\begin{equation}
\label{eq:taylor-min-max}
\begin{split}
\half\Bigg\{\max_{y\in \ol B_{\eps}(x)} \phi (y) &+
\min_{y\in \ol B_{\eps}(x)} \phi (y)\Bigg\} - \phi(x)
\\&
\geq\half \langle
D^2\phi(x)(x_1^{\varepsilon}-x),(x_1^{\varepsilon}-x)\rangle+o(\varepsilon^2).
\end{split}
\end{equation}
The computation leading to (\ref{exp.lapla}) gives
\begin{equation}
\label{eq:taylor-mean} \kint_{B_\varepsilon (x)} \phi
(y)
\ud y-\phi(x) =\frac{\eps^2}{2(n+2)}\Delta
\phi(x)+o(\varepsilon^2).
\end{equation}
 Then we multiply \eqref{eq:taylor-min-max} by $\alpha$ and
\eqref{eq:taylor-mean} by $\beta$ respectively, and add the formulas in the same way as we did in the sketch in the introduction.
We arrive at  the  expansion valid for any smooth function $\phi$:
\begin{equation}
\label{eq:asymp-p-lap}
\begin{split}
&\frac{\alpha}{2}\left\{\max_{y\in \ol B_{\eps}(x)}
\phi (y) + \min_{y\in \ol B_{\eps}(x)} \phi
(y)\right\}+ \beta\kint_{B_\varepsilon (x)} \phi (y)
\ud y -
\phi(x)\\
&\ge \frac{\beta\eps^2}{2(n+2)}\left((p-2)\Big\langle
D^2\phi(x)\left(\frac{x_1^{\varepsilon}-x}{\varepsilon}\right),
\left(\frac{x_1^{\varepsilon}-x}{\varepsilon}\right)\Big\rangle
+
\Delta \phi(x)\right)\\
& \quad +o(\varepsilon^2).
\end{split}
\end{equation}

Suppose that $
\phi$ touches the limit $u$ at $x$ from below
and that $D \phi(x)\neq 0$. Recall that it is enough to test with such
functions. By the uniform convergence, there exists a sequence
$x_{\eps} $ converging to $x$ such that $u_{\eps} - \phi $
has a minimum at $x_{\eps}$ (see \cite[Section 10.1.1]{evans10}), that is, there exists $x_{\eps}$ such that
$$
u_{\eps} (y) - \phi (y) \geq u_{\eps} (x_{\eps}) - \phi
(x_{\eps})
$$
at the vicinity of $x_{\eps}$.
As a matter of fact we omitted for simplicity an arbitrary constant $\eta_{\eps}>0$ that should be included due to the fact that $u_{\eps}$ may be discontinuous and we might not attain the minimum.
Moreover, by adding a constant, we may assume that $\phi
(x_{\eps}) = u_{\eps} (x_{\eps})$.
 Thus, by recalling the fact that $u_\eps$ is $p$-harmonious, we obtain
\[
\begin{split}
 0 \geq-\phi (x_{\eps})&+
\frac{\alpha}{2} \left\{ \max_{\ol B_{\eps} (x_{{\eps}})} \phi
+ \min_{\ol B_{\eps} (x_{{\eps}})} \phi
\right\}+ \beta
\kint_{B_{\eps} (x_{{\eps}})} \phi (y)\ud y.
\end{split}
\]
Combining this with \eqref{eq:asymp-p-lap} redefining the notation $x_1^\varepsilon$ as
\begin{align*}
\phi(x_1^\varepsilon)=\min_{y\in
\ol B_{\eps}(x_{\eps})}\phi(y),
\end{align*} 
we get
\begin{align}
\label{eq:RHS-expansion}
o&(\eps^2)\\
&\ge \frac{\beta{\eps}^2}{2(n+2)}\left((p-2)\Big\langle
D^2\phi(x_{\eps})\left(\frac{x_1^{{\eps}}-x_{{\eps}}}{{\eps}}\right),
\left(\frac{x_1^{{\eps}}-x_{{\eps}}}{{\eps}}\right)\Big\rangle
+
\Delta \phi(x_{{\eps}})\right).\nonumber
\end{align}
Since $D\phi(x)\not=0$, it follows from Taylor's expansion\footnote{Apart from terms that converge to zero we have $\frac{\phi(x_1^\varepsilon)-\phi(x_\varepsilon)}{\eps}\approx \langle D\phi(x_\varepsilon),  \frac{x_1^\varepsilon-x_\varepsilon}{\eps}\rangle$, where $x_1^\varepsilon$ is the minimum point. This implies the desired convergence.} using regularity of $\phi$ that
\begin{align*}
%\label{eq:limit}
 \lim_{\varepsilon \to
0}\frac{x_1^\varepsilon-x_\varepsilon}{\varepsilon} = -\frac{D
\phi(x)}{|D \phi(x)|} .
\end{align*}
Dividing (\ref{eq:RHS-expansion}) by $\eps^2$ and letting $\eps \to 0$, we end up with
\[
\begin{split}
0 &\geq
  \frac{\beta}{2(n+2)}\left((p-2) \Delta^N_{\infty}\phi(x)
+\Delta \phi (x)\right).
\end{split}
\]
Therefore $u$ is a viscosity supersolution. The proof for the subsolution property is similar, and thus we have that $u$ is a viscosity solution.

\subsection{Time dependent values and the normalized $p$-parabolic equation}
\label{sec:parab-convergence}

This section is based on \cite{manfredipr10c} and \cite{luirops14}.
Let
\begin{align*}
\Om_\infty=\Om\times (0,\infty)
\end{align*}
 denote a space-time cylinder. To prescribe boundary values, we
denote the parabolic boundary strip by
\[
\begin{split}
\Gamma_p^\eps= \Big(\Gamma_{\eps}\times (-\frac{\eps^2}{2},\infty) \Big)\cup
\Big(\Om\times (-\frac{\eps^2}{2},0] \Big),
\end{split}
\]
where $ \Gamma_\eps:=\{x\in \RR^n \setminus \Om\,:\,\dist(x, \Om) \leq \eps\}.$
Boundary values are given by $F: \Gamma_p^\eps\to \R$ which is a bounded Borel function. 

\subsubsection{Dynamic programming principle and value functions}

As we suggested before, the normalized $p$-parabolic functions
\begin{align*}
(n+p) u_t=\Delta_p^N u
\end{align*}
 are related to the game values of the time tracking tug-of-war with noise. We also suggested that the game values $u_{\eps}$ satisfy the parabolic dynamic programming principle (DPP) with
boundary values $F$ 
\[
\begin{split}
u_\eps(x,t)&= \frac{\alpha}{2}
\left\{ \sup_{y\in  B_\eps (x)} u_\eps \Big(y,t-\frac{\eps^2}{2}\Big) +
\inf_{y\in  B_\eps (x)} u_\eps\Big(y,t-\frac{\eps^2}{2}\Big) \right\}
\\
&\hspace{1 em}+\beta \kint_{B_\eps (x)} u_\eps\Big(y,t-\frac{\eps^2}{2}\Big)
\ud y 
\end{split}
\]
for every $(x, t) \in \Omega_\infty$, and 
\[
\begin{split}
u_\eps(x,t) &= F(x,t), \qquad \textrm{for every}\quad (x,t) \in
\Gamma_p^\eps.
\end{split}
\]
Moreover, recall that $2\le p<\infty$ and
\begin{align*}
\alpha=\frac{p-2}{p+n}, \qquad
\beta=\frac{2+n}{p+n}.
\end{align*}

The proof is simpler than in the elliptic case, since finite iteration is sufficient. All the functions below are Borel measurable on each time slice $\{(x,s):x\in \Om\cup \Gamma_{\eps},s=t\}$. The proof is from  \cite[Theorem 5.2]{luirops14}.
We define the operator $T$ as
\[
\begin{split}
Tv(x,t)=\begin{cases}
&\frac{\alpha}{2}
\left\{\sup_{y\in  B_\eps (x)} v\Big(y,t-\frac{\eps^2}{2}\Big) +
\inf_{y\in  B_\eps (x)} v\Big(y,t-\frac{\eps^2}{2}\Big) \right\}
\\
&\hspace{1 em}+\beta \kint_{B_\eps (x)} v\Big(y,t-\frac{\eps^2}{2}\Big)
\ud y \qquad \mbox{for every } (x, t) \in \Omega_\infty
\\
&F(x,t), \qquad \textrm{for every}\quad (x,t) \in
\Gamma_p^\eps.
\end{cases}
\end{split}
\]

\begin{theorem}
\label{thm:dpp-parab}
Given  bounded boundary values $F:\Gamma_p^\eps\to \R$, there exists a unique solution to the parabolic DPP.
\end{theorem}

\begin{proof}
Let
$$
u_0(x,t)=\begin{cases} 0,& (x,t)\in \Omega_{\infty}, \\
F(x,t), &(x,t)\in \Gamma_p^{\eps}.
\end{cases}
$$
We claim that the iteration  $u_{i+1}=Tu_i$, $i=0,2,\ldots$ converges in a finite number of steps at any fixed $(x,t)\in \Om_\infty$. To establish this, we use induction, and show that when calculating the values $$u_{i+1}=Tu_i$$ only the values for $t>i\eps^2/2$ can change. This is clear if $i=0$ since then the operator $T$ uses the values for  $t\in (-\eps^2/2,0]$ which are given by $F$. Suppose then that this holds for some $k$ i.e.\ only the values $t>k\eps^2/2$ of $u_k$ can change in the operation $Tu_k$. Then from
\[
\begin{split}
Tu_{k+1}= T(Tu_k)
\end{split}
\]
we deduce that $T$ can only change the values for $t>(k+1)\eps^2/2 $ for $u_{k+1}$. Clearly, the limit satisfies the DPP. Moreover, the uniqueness follows by the same induction argument.   
\end{proof}

The rules of the game and the related notation can be recalled from Section \ref{sec:intro-time-dependent}. 
The value for Player I when starting at $x_0$ with remaining time $t_0$ is defined as  
\[
u_\I^{\eps}(x_0,t_0)=\sup_{S_{\I}}\inf_{S_{\II}}\,
\mathbb{E}_{S_{\I},S_{\II}}^{x_0,t_0}[F(x_{\tau},t_{\tau})],
\]
while the value for Player II is given by
\[
u^{\eps}_\II(x_0,t_0)=\inf_{S_{\II}}\sup_{S_{\I}}\,
\mathbb{E}_{S_{\I},S_{\II}}^{x_0,t_0}[F(x_{\tau},t_{\tau})].
\]
That the game has the unique value
then follows in a similar way as for time independent values in Section \ref{sec:elliptic-value}.

\begin{theorem}\label{th:value-parab}
It holds that $u=u^\eps_\I = u^\eps_\II$, where $u$ is the solution of the DPP
obtained in Theorem \ref{thm:dpp-parab}. 
\end{theorem}

\subsubsection{Convergence of value functions to a normalized or game theoretic $p$-parabolic function}

Recall the theorem we stated earlier as Theorem \ref{thm:sol-to-p-parabolic}.
\begin{theorem}
\label{thm:sol-to-p-parabolic2}
Let $\Om$ be a smooth domain, $F:\R^{n+1}\setminus\Om_T\to \R$ a smooth function, and $u$ the unique viscosity solution to the normalized $p$-parabolic boundary value problem
\[
\begin{split}
\begin{cases}
(n+p)u_t-\Delta_p^N u=0 & \text{ in }\Om \times (0,T]\\
 u=F& \text{ on }\partial_p \Om_T.
\end{cases}
\end{split}
\]
Further, let $u_{\eps}$ be the value function for the time tracking tug-of-war with noise with the boundary payoff function $F$. Then 
\[
\begin{split}
u_{\eps} \to u \text{ uniformly on } \ol \Om\times [0,T],
\end{split}
\]
as $\eps\to 0$.
\end{theorem}

%%%%%%%%%%
\begin{proof}
We omit the regularity, convergence and considerations at the latest moment, and only concentrate on the viscosity solution property of the limit for $(x,t)\in \Om_T$.
For $\phi\in C^2$, we select $x_1^{\eps,t-\eps^2/2}\in \ol B_\eps(x)$ such that 
\[
\begin{split}
\phi\Big(x_1^{\eps,t-\eps^2/2},t-\frac{\eps^2}{2}\Big)=
\min_{y\in \ol B_\eps(x)} \phi\Big(y,t-\frac{\eps^2}{2}\Big).
\end{split}
\]
Similarly as above in the elliptic case in (\ref{eq:asymp-p-lap}), we can derive 
an estimate
\begin{equation}
\label{eq:parab-taylor-comb-pointwise}
\begin{split}
 \frac{\alpha}{2}&
\Bigg\{\max_{y\in \ol B_\eps (x)} \phi \Big(y,t-\frac{\eps^2}{2}\Big) +
\min_{y\in \ol B_\eps(x)} \phi \Big(y,t-\frac{\eps^2}{2}\Big)\Bigg\}\\
&\hspace{5 em}+\beta \kint_{B_\eps (x)}
\phi\Big(y,t-\frac{\eps^2}{2}\Big) \ud y  - \phi(x,t)
\\&
\geq \frac{\beta \eps^2}{2(n+2)}\Bigg(
(p-2)\left\langle
D^2\phi(x,t)\frac{x_1^{\eps,t-\eps^2/2}-x}{\eps},
\frac{x_1^{\eps,t-\eps^2/2}-x}{\eps}\right\rangle
\\
&\hspace{7 em}+\Delta \phi(x,t)-(n+p)\phi_t(x,t) \Bigg)+o(\eps^2).
\end{split}
\end{equation}

Suppose then that $
\phi$ touches the limit $u$ at $(x,t)$  from below. By the uniform convergence, there exists a sequence
$(x_{\eps},t_\eps) $ converging to $(x,t)$ such that $u_{\eps} - \phi $
has an (approximate) minimum at $(x_\eps,t_\eps)$.  More precisely, again omitting small constant $\eta_{\eps}>0$ due to discontinuity, there exists $(x_\eps,t_\eps)$ such that
$$
u_{\eps} (y,s) - \phi (y,s) \geq u_{\eps} (x_\eps,t_\eps) - \phi
(x_\eps,t_\eps) 
$$
at the vicinity of $(x_\eps,t_\eps)$. 
Further, by adding a constant, we may assume that
\[
\begin{split}
u_{\eps} (x_\eps,t_\eps) = \phi
(x_\eps,t_\eps)\quad \text{and}\quad  u_{\eps} (y,s)\geq \phi(y,s).
\end{split}
\]
Recalling the fact that $u_\eps$ is the game value and thus satisfies the parabolic DPP, we obtain
\begin{equation}
\label{eq:asymp-expansion-almost-negative}
\begin{split}
0 \geq &-\phi(x_\eps,t_\eps)+\beta \kint_{B_\eps (x_\eps)}
 \phi\Big(y,t_\eps-\frac{\eps^2}{2}\Big)
\ud y\\
&+ \frac{\alpha}{2}
\left\{ \sup_{y\in \ol B_\eps (x_\eps)} \phi\Big(y,t_\eps-\frac{\eps^2}{2}\Big) +
\inf_{y\in \ol B_\eps (x_\eps)} \phi\Big(y,t_\eps-\frac{\eps^2}{2}\Big) \right\}.
\end{split}
\end{equation}
According to \eqref{eq:parab-taylor-comb-pointwise},  we have
\[
\begin{split}
0\geq \frac{\beta \eps^2}{2(n+2)}&\Bigg(
(p-2)\left\langle
D^2 \phi(x_\eps,t_\eps)\frac{x_1^{\eps,t_\eps-\eps^2/2}-x_\eps}{\eps},
\frac{x_1^{\eps,t_\eps-\eps^2/2}-x_\eps}{\eps}\right\rangle
\\
&\hspace{1 em}+\Delta  \phi(x_\eps,t_\eps)-(n+p) \phi_t(x_\eps,t_\eps) \Bigg)+o(\eps^2),
\end{split}
\]
where $x_1^{\eps,t_\eps-\eps^2/2}$ is now the minimum point related to $(x_{\eps},t_{\eps})$.
Suppose that $ D \phi(x,t)\neq 0$. Dividing by $\eps^2$ and
letting $\eps \to 0$, we get
\[
\begin{split}
0 &\geq
  \frac{\beta}{2(n+2)}\big((p-2) \Delta^N_{\infty} \phi(x,t)
+\Delta \phi (x)-(n+p) \phi_t(x,t)\big).
\end{split}
\]
To verify the other half of the definition of a viscosity solution, we derive a
reverse inequality to
\eqref{eq:parab-taylor-comb-pointwise} by considering the maximum point of the
test function  and choose a function $\phi$ which touches $u$ from
above. The rest of the argument is similar.

In the parabolic case, we also need to consider $ D \phi(x,t)=0$. By
Theorem \ref{thm:test-functions-neglect}, we can also assume that
$D^2 \phi(x,t)=0$ and it suffices to show that
\[
\begin{split}
\phi_t(x,t)\geq 0.
\end{split}
\]
Since in this case
\begin{align*}
\left\langle
D^2 \phi(x_\eps,t_\eps)\frac{x_1^{\eps,t_\eps-\eps^2/2}-x_\eps}{\eps},
\frac{x_1^{\eps,t_\eps-\eps^2/2}-x_\eps}{\eps}\right\rangle\to 0
\end{align*}
as $\eps \to 0$ on the right hand side of \eqref{eq:parab-taylor-comb-pointwise}, and  \eqref{eq:asymp-expansion-almost-negative} still holds, dividing \eqref{eq:parab-taylor-comb-pointwise} by $\eps^2$ and passing to the limit gives $\phi_t(x,t)\ge 0$.
\end{proof}

Observe that the above theorem also gives a probabilistic
proof for the existence of viscosity solutions to $(n+p)u_t=\Delta_p^N u$.

%%%%%%%%%%%%%%%%%%%%%%%%%%%%%%%%%%%%%%%%%%%%%%%%%%%%%%%%%%%%%%%%

\section{Cancellation method for regularity of the tug-of-war with noise}
\label{sec:Lip-regularity}

In this section, we show that a value function for the tug-of-war with noise is (asymptotically) locally Lipschitz continuous. We follow \cite{luirops13}.

By passing to the limit with the step size, we obtain a new and direct game theoretic proof for Lipschitz continuity of
$p$-harmonic functions with $p>2$.
The proof  is based on the suitable choice of strategies, and is thus
completely different from the proofs based on the works of De Giorgi, Moser, or
Nash.

\subsection{Local Lipschitz estimate}
\label{sec:lipschitz}

We start by giving an example of stochastic approach in the regularity theory.
%%%%%%%%%%
\begin{example}[Simple proof of Harnack by playing games]
\label{ex:simple-harnack}
\idea If a player can guarantee to reach a single point with a
uniform positive probability before the game position is too far away, then regularity properties should follow.

Let $u_{\eps}\ge 0$ be a value function for the tug-of-war with noise defined in a large enough domain. Then let
 $x_0,y_0\in \overline B_1(0)$ such that
\begin{align*}
u_{\eps}(x_0)=\inf_{B_1(0)} u_{\eps} \quad \text{ and }\quad u_{\eps}(y_0)= \sup_{B_1(0)} u_{\eps}.
\end{align*}
We omit small constants by assuming above that we attain $\inf$ and $\sup$.

 We start the game at $x_0$ and stop if we reach $y_0$ or exit $B_2(0)$. Player I tries to reach the point $y_0$. Assume that there is a uniform (independent of $\eps$) lower bound $P>0$ for the probability of reaching $y_0$ before exiting  $B_2(0)$. Then we can estimate
 \begin{align*}
\inf_{B_1(0)} u_{\eps} = u_{\eps}(x_0)\ge P u_{\eps}(y_0)+(1-P)\cdot 0=P u_{\eps}(y_0) = P\sup_{B_1(0)} u_{\eps}.
\end{align*}
Thus we have shown Harnack's inequality 
$$
\sup_{B_1(0)} u_{\eps}\le C \inf_{B_1(0)} u_{\eps},
$$ 
which has the well-known connection to the Hölder continuity. 

Unfortunately, the desired uniform $P$ only exists when $p>n$ (recall that probabilities in the tug-of-war were given in terms of the parameter $p$). If $p\le n$, 
a player cannot force the game position to a point with uniformly positive probability and we need to come up with another approach.
\end{example}
%%%%%%%%%%

{\sc Idea} of the  proof  of the Lipschitz estimate: 
\begin{itemize}
\item We want to estimate the difference of values of game instances started at two different points.
\item  Use a {\bf cancellation strategy}: We try to cancel all the steps of the opponent, and make a certain translation (good case). If we can do this,  we can focus on random steps. 
\item Once we complete the cancellation starting at the two different points, then there will be a cancellation in the game values (cancellation is used in \eqref{eq:key-cancellation} below)
\item It remains to estimate the probability of the good case by a specialized random walk called {\bf 'cylinder walk'} (cylinder walk is a linear problem).
\end{itemize}

Also observe that we need to include coin tosses to the history of the game in order to run the cancellation strategy i.e.\ to know the previous steps of the opponent. Nonetheless, earlier proofs work in the same way as before with such a history. We also assume that $p>2$ implying $\alpha>0$ so that we have players and strategies to work with. 

\begin{theorem}[Lipschitz estimate]
\label{thm:lipschitz-continuity}
Let $u_{\eps}$ be a value function in $\Om$ for the tug-of-war with noise with $\alpha>0$, assume that $B_{10r}(z_0)\subset\Om$ and $r>\varepsilon .$ Then
\[
|u_{\eps}(x)-u_{\eps}(y)|\leq L |x-y|(\sup_{B_{6r}(z_0)} u_{\eps}-\inf_{B_{6r}(z_0)} u_{\eps})/r
\]
for all  $x,y\in B_{r}(z_0)$ with $|x-y|\geq\eps$.
\end{theorem}
\begin{proof}
Let $x,y\in B_r(z_0)$ and $2m:=\abs{x-y}/\eps$. Fix $z\in B_{2r}(z_0)$ so that 
$$
|z-x|=|z-y|=m \eps.
$$ 
We define a strategy $S^0_{\II}$ for Player II for the game that starts at $x$: he always tries to cancel the earliest move of Player I which  he has not yet been able to cancel. If all the moves at that moment are cancelled and he wins the coin toss, then he takes the step\footnote{Here we made two simplifications: actually the players cannot quite take steps of size $\eps$ %since they choose the next point from an open ball, 
and $\abs{x-y}/\eps$ is not necessarily an even integer $2m$: nonetheless, these issues could be treated by small corrections.} 
\begin{align*}
\eps \frac{z-x}{|z-x|}.
\end{align*}

At every moment we can divide the game position as a sum of vectors
\[
x+\sum_{k\in I_1}x^1_k+\sum_{k\in I_2}x^2_k+\sum_{k\in I_3}h_k\,.
\]
Here $I_1$ denotes indices of rounds when Player I has moved, vectors $x^1_k$ are her moves, and  $x^2_k$ represent the moves of Player II.
Set $I_3$ denotes indices for random moves, and these vectors are denoted by $h_k$.

Then we define a stopping time ${{\tau^*}}$ for this process. We give three conditions to stop the game:
\begin{enumerate}[(i)]
\item  \label{stop-i} If Player II has won $m$ more turns than Player I. (Good case)
\item \label{stop-ii} If Player I has won at least $\frac{2r}{\eps}$ turns more than Player II.
\item \label{stop-iii} If $|\sum_{k\in I_3}h_k|\geq 2r$.
\end{enumerate}
This stopping time is finite with probability $1$, and does not depend on strategies nor starting points. 
Let us then consider the situation when the game has ended by condition (\ref{stop-i}). In this case,
the sum of the steps made by players  can easily  be computed, and it is
\[
\begin{split}
\sum_{k\in I_1}x^1_k+\sum_{k\in I_2}x^2_k=m \frac{\eps(z-x)}{|z-x|}=z-x\,.
\end{split}
\]
Thus the stopping point $x_{{\tau^*}}$ is actually randomly chosen around $z$: in case (\ref{stop-i}), we have
\begin{equation}
\label{eq:xtau}
\begin{split}
x_{{\tau^*}}=z+\sum_{k\in I_3}h_k\,.
\end{split}
\end{equation}
These conditions also guarantee that we never exit $B_{6r}(z_0)$.

A crucial point is the following cancellation effect. Let $S^0_I$ be the corresponding cancellation strategy for Player I when starting at $y$. Then
\begin{equation}
\label{eq:symmetry}
\begin{split}
\mathbb{E}^x_{S_{\I},S^0_{\II}}&[u_{\eps}(x_{{\tau^*}})\,|\trm{ game ends by condition (i)}]\\
=&\mathbb{E}^{y}_{S^0_{\I},S_{\II}}[u_{\eps}(x_{{\tau^*}})\,|\trm{ game ends by condition (i)}]
\end{split}
\end{equation}
for any choice of the strategies $S_{\I}$ or $S_{\II}$. Indeed, if the game ends by the condition (i), then $x_{{\tau^*}}$ in \eqref{eq:xtau} has only the random part which is independent of the strategies and starting points.

Let us denote by $P$ the probability that the game ends by the condition (i). 
Then assuming without loss of generality that $u_{\eps}(x)\ge u_{\eps}(y)$ and also using a localization result, Lemma~\ref{lem:value-estimate} below, we get
\begin{align}
\label{eq:key-cancellation}
|u_{\eps}(x)-u_{\eps}(y)|&\stackrel{\text{Lemma~\ref{lem:value-estimate} }}{\le}\abs{\sup_{S_{\I}}\mathbb{E}^x_{S_{\I},S^0_{\II}}[u_{\eps}(x_{{\tau^*}})]
-\inf_{S_{\II}}\mathbb{E}^y_{S^0_{\I},S_{\II}}[u_{\eps}(x_{{\tau^*}})]}\nonumber \\
&\hspace{1.5 em}\le P\Big|\sup_{S_{\I}}\mathbb{E}^x_{S_{\I},S^0_{\II}}[u_{\eps}(x_{{\tau^*}})\,|\,\trm{(i)}]-\inf_{S_{\II}}\mathbb{E}^y_{S^0_{\I},S_{\II}}[u_{\eps}(x_{{\tau^*}})\,|\,\trm{(i)}]\Big|\nonumber \\
&\hspace{2.5 em}+(1-P)(\sup_{B_{6r}(z_0)}u_{\eps}-\inf_{B_{6r}(z_0)}u_{\eps})\nonumber \\
&\hspace{1.3 em}\stackrel{\text{(\ref{eq:symmetry})}}{\le} (1-P)(\sup_{B_{6r}(z_0)}u_{\eps}-\inf_{B_{6r}(z_0)}u_{\eps}),
\end{align}
where we shortened 'game ends by condition (i)' to '(i)', and also used the fact that we never exit $B_{6r}(z_0)$. 

It remains to verify that $1-P$ is small enough. To this end, observe that with the choice $\ell=m\varepsilon$ the  estimate for the `cylinder walk' in Lemma \ref{lemma:speed} below gives the upper  bound for $1-P$ in the form
$$
1-P\leq  C(p,n)m\varepsilon/r\leq L|x-y|/r
$$
for $|x-y|\geq \varepsilon .$
\end{proof}

The following observation was used above in localizing the estimates and also in the Harnack example.
\begin{lemma}
\label{lem:value-estimate}
Let $\tau^*$ be a stopping time with $\tau^*\leq\tau $, where $\tau$ is the usual exit time from the domain. Then
\[
\begin{split}
u_{\eps}(y)\leq\sup_{S_{\I}}\mathbb{E}^y_{S_{\I},S^0_{\II}}[u_{\eps}(x_{\tau^*})],
\end{split}
\]
for any fixed $S^0_\II$. Similarly
\[
\begin{split}
u_{\eps}(y)\geq\inf_{S_{\II}}\mathbb{E}^y_{S^0_{\I},S_{\II}}[u_{\eps}(x_{\tau^*})],
\end{split}
\]
for any fixed $S^0_\I$.
\end{lemma}
%%%%%%%%%%%%%%%%%%%
\begin{proof}
We only prove the first statement; the proof of the second statement is similar. Assume that we are given a fixed strategy $S^0_\II$. Let then
Player I  follow  a strategy
$S^\trm{max}_\I$ such that at $x_{k-1}\in \Om$ she chooses to step to a
point that almost maximizes $u_\eps$, that is, to a point $x_k \in B_\eps (x_{k-1})$ such that
\[
u_\eps(x_k)\ge \sup_{ B_\eps(x_{k-1})} u_\eps-\eta 2^{-k}
\]
for some fixed $\eta>0$. By this choice and the DPP, we get
\begin{align}
\label{eq:key-mart}
&\mathbb{E}_{S^\trm{max}_\I, S^0_\II}^{x_0}[u_\eps(x_k)-\eta 2^{-k}\,|\mathcal F_{k-1}] \nonumber
\\
&\ge \frac{\alpha}{2} \left\{\inf_{ B_\eps(x_{k-1})} u_\eps+\sup_{  B_\eps(x_{k-1})}
u_\eps-\eta 2^{-k}\right\}+ \beta \kint_{ B_{\eps}(x_{k-1})}
u_\eps \ud y-\eta 2^{-k}\nonumber
\\&\ge  u_\eps(x_{k-1})-\eta 2^{-(k-1)}.
\end{align}
Above $\mathcal F_{k-1}$ denotes the history of the game. Earlier we denoted this by $(x_0,\ldots,x_{k-1})$ but now we also need to include the coin tosses to be able to run the cancellation strategy, and opted for this shorthand. 

From the above it follows that  $M_k=u_\eps(x_k)-\eta 2^{-k}$ is a bounded  submartingale.
By the optional stopping theorem,  it follows that
\[
\begin{split}
\sup_{S_{\I}}\mathbb{E}_{S_{\I},S^0_{\II}}^{y}[u_\eps(x_{\tau^*})]&\ge  \mathbb{E}_{S^\trm{max}_\I, S^0_\II}^{y}[u_\eps(x_{\tau^*})]\\
&\ge \mathbb{E}_{S^\trm{max}_\I, S^0_\II}^{y}[u_\eps(x_{\tau^*})-\eta 2^{-\tau^*}]\\
&\ge \mathbb{E}^y_{S^\trm{max}_\I, S^0_\II}[M_0]=u_\eps(y)-\eta.
\end{split}
\]
Since $\eta>0$ is arbitrary, this completes the proof.
\end{proof}

\subsection{Cylinder walk}
\label{sec:cylinder-walk}
At the end of the proof of Theorem \ref{thm:lipschitz-continuity}, we needed an estimate for $P$, the probability that one of the players gets a surplus of $m$ wins in the coin toss before other stopping conditions take effect.

To start, consider the following random walk (called the cylinder walk) in  a $n+1$ -dimensional cylinder $B_{2r}(0)\times (0,2r+\ell)$, $B_{2r}(0)\subset \R^n$ as in Figure~\ref{fig:cylinderwalk}. Suppose that we are at a point 
$$
(x_j,y_j)\in B_{2r}(0)\times (0,2r+\ell).
$$ 

\begin{itemize}
\item With probability $\alpha/2$ we move to the point $(x_{j+1},y_{j+1}):=(x_j,y_{j}-\eps)$.
\item With probability $\alpha/2$ we move to the point $(x_{j+1},y_{j+1}):=(x_j,y_{j}+\eps)$.
\item With probability $\beta$ we move to the point $(x_{j+1},y_{j+1}):=(x_{j+1},y_j)$, where $x_{j+1}$ is randomly chosen from the ball $B_{\eps}(x_j)$ according to the uniform probability distribution.
\end{itemize}
Moreover, we stop when exiting the cylinder and denote by $\tau'$ the exit time. 

\idea If we start at $(0,\ell)=(0,m\eps)$, exiting through different parts of the boundary corresponds to the stopping rules  (\ref{stop-i})-(\ref{stop-iii}) above. In particular, the good case corresponds exiting through $B_1(0) \times \{0\}.$

\begin{figure}[h]
 \begin{tikzpicture}[scale=0.7]
 %Cylinder itself
 \draw (0,6.25) ellipse (4 and 1.6);
 \draw (-4,0) arc (180:360:4 and 1.6);
 \draw[dashed] (4,0) arc (0:180:4 and 1.6);
 \draw (4,6.25) -- (4,0);
 \draw (-4,6.25) -- (-4,0);
 %The point and its possible movements
 \filldraw[black] (0,2) circle (1.5pt);
 \filldraw[gray] (0,2.9) circle (1.5pt);
 \filldraw[gray] (0,1.1) circle (1.5pt);
 \draw[thick,->] (0,2.1) -- (0,2.8);
 \draw[thick, ->] (0,1.6) -- (0,1.2);
 \draw[densely dashed] (0,1.9) -- (0,1.7);
 \draw (0,2) ellipse (0.9 and 0.4);
 \draw (1.4,2.4) -- (0.1,2.03);
 \node[scale=0.9] at (2,2.4) {$(0,\ell)$};
 \node[scale=0.9] at (-1.1,3.2) {$(0,\ell+\varepsilon)$};
 \node[scale=0.9] at (1.1,0.8) {$(0,\ell-\varepsilon)$};
 \node[scale=0.9] at (-1.6,2.1) {$B_\varepsilon(0)$};
 %Bottom
 \draw (0,0) -- (-4,0);
 \node[scale=0.9] at (-2.2,0.25) {$2r$};
 \filldraw[black] (0,0) circle (1.5pt);
 %Bracket
 \draw[decorate,decoration={brace,mirror,amplitude=15pt}] (4.3,0) -- (4.3,6.25);
 \node at (6.1,3.12) {$2r+\ell$};
 %Boundary values
 \node at (0,-1.9) {$1$};
 \node at (0,6.25) {$0$};
 \node at (4.2,3.12) {$0$};
 \node at (-4.2,3.12) {$0$};
 \end{tikzpicture}
\caption{The initial setting of the cylinder walk. The boundary values $0$ or $1$ are selected in the proof of Lemma~\ref{lemma:speed}.}
\label{fig:cylinderwalk}
\end{figure}
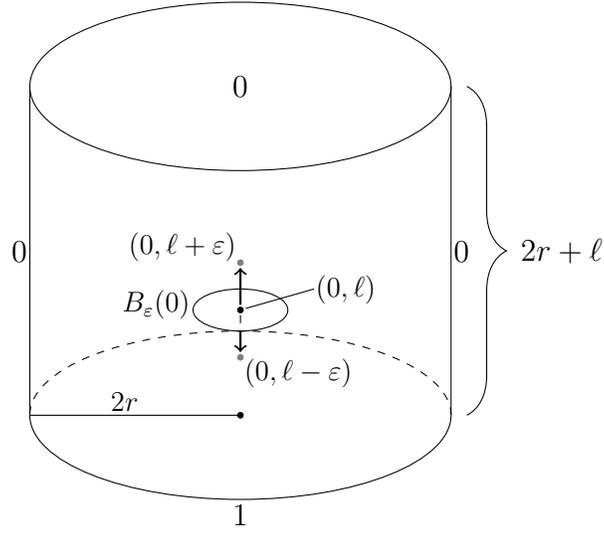

The next lemma gives an estimate for the probability that the cylinder walk escapes through the bottom. In \cite{luirops13}, there is a direct stochastic proof, but here we give a PDE oriented proof.
\begin{lemma}
\label{lemma:speed}
Let us start the cylinder walk at the point $(0,\ell)$. Then the probability that the walk does not escape the cylinder through its bottom is less than
$$
L(\ell+\varepsilon)/r,
$$
for all $\eps>0$ small enough.
\end{lemma}
\idea To estimate the escape probability through the bottom, we set the boundary value $1$ at the bottom and $0$ on the other parts of the boundary as in Figure~\ref{fig:cylinderwalk}, and estimate the expected value of the cylinder walk, cf.\ (\ref{eq:hitting-prob}). The cylinder walk reminds the usual random walk and in particular satisfies the dynamic programming principle of the form
\begin{align*}
u(x,y)=\beta \kint_{B_{\eps}(0)}u(x+h,y)\ud h+\frac{\a}{2}\{u(x,y-\eps)+u(x,y+\eps)\}.
\end{align*}
This reminds the usual mean value property for harmonic functions, and thus the value $\mathbb E[u(x_{\tau'},y_{\tau'})]$ of the cylinder walk is close to a modified harmonic function with the above boundary values.
%%%%%%%%%%%%%%%%%%%%%%%%%%%
\begin{proof}[Proof of Lemma~\ref{lemma:speed}.]
{\bf Step 1 (exit probability through the bottom):} Let us  outline the proof for Lemma \ref{lemma:speed} in $U:= B_1(0)\times (0,1)$ for simplicity. 
Assume for now that we have a smooth $u$ that satisfies the asymptotic expansion 
\begin{align*}
u(x,y)=\beta \kint_{B_\eps(0)} u(x+h,y) \ud h&+\frac{\alpha}{2}\{u(x,y+\eps)+u(x,y-\eps)\} + O(\varepsilon^3),
\end{align*}
with a uniform error $\abs{O(\varepsilon^3)}\le C\eps^3$, and has boundary value $1$ on $B_1(0)\times \{0\}$ and $0$ elsewhere.\footnote{When later giving explicit $u$, it is defined in a larger set to take technical issues with the discrete step size into account.}

Consider the sequence of random variables $u(x_j,y_j)$, $j=0,1,2,\ldots ,$ where
$(x_j,y_j)_{j\in\N}$ are the positions in the cylinder walk with $(x_0,y_0)=(0,\ell)$. 
Observe that the above asymptotic expansion makes 
$
M_j:=u(x_j,y_j)+cj\varepsilon^3
$ a submartingale (see (\ref{eq:key-mart}) for the notation):
\begin{align*}
&\mathbb{E}[u(x_j,y_j)+cj\varepsilon^3\,|\mathcal F_{j-1}] \nonumber
\\
&=\beta \kint_{B_\eps(0)} u(x_{j-1}+h,y_{j-1}) \ud h\\
&\hspace{5 em}+\frac{\alpha}{2}\{u(x_{j-1},y_{j-1}+\eps)+u(x_{j-1},y_{j-1}-\eps)\} +cj\varepsilon^3\\
&=u(x_{j-1},y_{j-1})-O(\eps^3)+cj\varepsilon^3\\
&\ge u(x_{j-1},y_{j-1})+c(j-1)\varepsilon^3,
\end{align*}
for suitable $c$.
Hence by the optional stopping theorem, we  deduce that
\begin{align*}
\mathbb E[M_{\tau'}]=\mathbb E[u(x_{\tau'},y_{\tau'})+c\tau' \eps^3]\ge M_0=u(0,\ell).
\end{align*}
Denote by $P$ the exit probability through $B_1(0)\times \{0\}$. Then by the above
\begin{align}
\label{eq:Lip-estimate}
u(0,\ell) -c\mathbb E[\tau']\eps^{3}\le  \mathbb E[u(x_{\tau'},y_{\tau'})]\le P\cdot 1+ (1-P)\cdot  0=P.
\end{align}

Next we show that  $0\le \E[ \tau']\leq C\varepsilon^{-2}$, cf.\ observation in Example \ref{ex:walk-laplace}.  This follows  from the fact that $M_j:=y^2_j-\alpha j\varepsilon^2$ is a martingale: We see the martingale property by using
$$
\frac{\alpha}{2}((  y_j+\varepsilon)^2+(  y_j-\varepsilon )^2)+\beta y_j^2 =  y_j^2+\alpha\varepsilon^2 .
$$
in the computation
\begin{align*}
\E[M_{j+1}\,|\, \mathcal F_{j}]&=\E[y^2_{j+1}-\alpha (j+1)\varepsilon^2\,|\, \mathcal F_{j}]\\
&=\alpha((  y_j+\varepsilon)^2+(  y_j-\varepsilon )^2)/2+\beta y_j^2-\alpha (j+1)\varepsilon^2\\
&=y_j^2+\alpha\varepsilon^2-\alpha (j+1)\varepsilon^2\\
&=y_j^2-\alpha j \varepsilon^2=M_j.
\end{align*}
By reverse Fatou's lemma  and the  optional stopping theorem
\begin{align*}
y_0^2=\E[M_0]= \limsup_{k\to \infty}\E[M_{\tau'\wedge k}]\le \E[M_{\tau'}] =\E[y^2_{\tau'}-\alpha \tau' \varepsilon^2],
\end{align*}
from which it follows that 
$$
y_0^2\le\E[(1+\eps)^{2}-\a \tau' \eps^{2}].
$$ 
By rearranging, we get 
\begin{align*}
 \E[ \tau']\le C\eps^{-2}
\end{align*}
as desired.
Combining this with (\ref{eq:Lip-estimate}), and recalling that $u$ is assumed to be smooth and thus Lipschitz continuous, we get
\begin{align*}
1-L(\ell+\eps) \le u(0,\ell) -c\mathbb E[\tau']\eps^{3}\le P
\end{align*}
Thus
\begin{align*}
1-P\le L(\ell+\eps), 
\end{align*}
which proves the claim once we find suitable $u$.
%%%%%%%%%%%%%%%%%%%%%%%%%%%%%%%%%%%%%%%%%%%%%%%%%%

{\bf Step 2 (finding $u$):} It remains to find a smooth $u$ satisfying the asymptotic expansion and taking suitable boundary values. Actually, we find a suitable solution $u\leq 1$  to a linear PDE 
$$
u_{x_1x_1}+\ldots+u_{x_n x_n}+(p-2)u_{yy}=0
$$
in a larger domain so that
\begin{align*}
u(x,-\eps)\approx& 1\quad\trm{when $x$ is near $0$},\\
u\leq& 0\quad  \text{on the other parts of the boundary.} 
\end{align*}
To construct a suitable solution, one can use a rescaled fundamental solution to the Laplace operator as explained later. 
Solution to the above PDE satisfies the desired asymptotic expansion: Consider the Taylor expansion
\[
u(x+h,y)\approx u(x,y)+\langle D u(x,y), h \rangle 
+\frac12 \langle D^2u(x,y)h , h\rangle 
\]
which implies by our earlier computations (see (\ref{exp.lapla}))
\begin{align*}
\kint_{B_{\eps}(0)}u(x+h,y)\ud h\approx u(x,y)+\frac{\eps^2}{2(n+2)}\sum_{i=1}^n u_{x_ix_i}(x,y).
\end{align*}
Moreover,
\[
u(x,y\pm\eps)\approx u(x,y)\pm u_y(x,y)\eps+\half  u_{yy}(x,y)\eps^2.
\]
We multiply the expansions by 
\begin{align*}
 \beta=\frac{n+2}{p+n} \quad \text{and} \quad \frac{\alpha}{2}=\frac{p-2}{2(p+n)}
\end{align*}
respectively, recall the error terms, sum up, and  get
\begin{align*}
\beta &\kint_{B_\eps(0)} u(x+h,y) \ud h+\frac{\alpha}{2}\{u(x,y+\eps)+u(x,y-\eps)\}\\
&=u(x,y)+\frac{\eps^2}{2(p+n)}\Big\{\sum_{i=1}^n u_{x_ix_i}(x,y) + u_{yy}(x,y) (p-2)\Big\}+O(\eps^3)\\
&=u(x,y)+O(\eps^3),
\end{align*}
since $u$ is a solution to the PDE.

A good enough explicit solution to the above PDE-problem can be obtained by scaling a harmonic function. 
Let $v$ be a harmonic function and denote $(x,y)=(x_1,\ldots,x_n,y)$. Then let $u(x,y)=v(x,y/\sqrt{p-2})=v(x,s)$ and thus
\begin{align*}
u_{yy}(x,y)&=\frac{v_{ss}(x,y/\sqrt{p-2})}{\sqrt{p-2}^{2}}.
\end{align*}
It follows that $u$ solves
\begin{align*}
u_{x_1x_1}&(x,y)+\ldots+u_{x_n x_n}(x,y)+(p-2)u_{yy}(x,y)\\
&=v_{x_1x_1}(x,s)+\ldots+v_{x_n x_n}(x,s)+(p-2)\frac{v_{ss}(x,y/\sqrt{p-2})}{\sqrt{p-2}^{2}}\\
&=v_{x_1x_1}(x,s)+\ldots+v_{x_n x_n}(x,s)+v_{ss}(x,s)\\
&=0.
\end{align*}
Thus a suitable function can be obtained by rescaling and translating a fundamental solution to the Laplace equation in such a way that the level set 1 passes through $(0,-\eps)$ and level set zero through the set $\partial B_1(0)\times \{0\}$.   
\end{proof}
The parabolic version is established in \cite{parviainenr16}.

\section{Mean value characterizations}

It is well known, and can be found in many PDE textbooks, that mean value property characterizes harmonic functions, i.e.\ roughly 
\begin{align*}
\Delta u=0\quad  \Leftrightarrow \quad u(x)=\kint_{B_{\eps}(x)} u \ud y.
\end{align*}
 In fact, we can relax
this condition by requiring that it holds asymptotically
$$
u(x) =
\frac{1}{|B_\varepsilon (x) |} \int_{B_\varepsilon (x) } u(y) \, dy + o(\varepsilon^2),
$$
as $\varepsilon\to 0$, by classical results of I. Privaloff (1925), W. Blaschke (1916), and S. Zaremba (1905).  
Interestingly a weak asymptotic mean value formula holds in the nonlinear case as well. The asymptotic mean value characterization was derived in \cite{manfredipr10}.
This section does not rely on games.

\subsection{Elliptic case}
Recall that 
\begin{equation}\label{alphaybeta}
\alpha=\frac{p-2}{p+n},\quad\text{and}\quad \beta=1-\alpha=\frac{2+n}{p+n}.
\end{equation}
In contrast to the previous sections, we assume $1<p\le \infty$, and we might also have negative $\alpha$. 

Recall that we defined viscosity solutions by touching with test functions such that $D\phi(x)\neq 0$ at a touching point (and we also require this when $p=\infty$).
 Also recall that by the remarks related to Definition~\ref{eq:def-normalized-visc}, which also hold when $1<p\le \infty$ by the references given there, the different definitions of viscosity solutions to $\Delta_p^N u=0$ or $\Delta_p u=0$, and $\Delta_\infty^N u=\abs{Du}^{-2}\langle D^{2}u Du,Du \rangle =0$ or $\Delta_{\infty} u=\langle D^{2}u Du,Du \rangle =0$ are equivalent. 

\begin{definition} \label{def.sol.viscosa.asymp}
A continuous function  $u$  satisfies
\begin{equation}\label{ref}
u(x) = \frac{\alpha}{2} \left\{ \max_{\overline{B}_\varepsilon
(x)} u +
\min_{\overline{B}_\varepsilon (x)} u
\right\} + \beta  \kint_{B_\varepsilon (x)}
u (y) \ud y+ o(\varepsilon^2), \quad \mbox{as }\varepsilon \to 0,
\end{equation}
\emph{in the viscosity sense} if
\begin{enumerate}
\item for every $ \phi\in C^{2}$ such that $\phi$ touches $u$ at the point $x \in \Omega$  from below and $D \phi(x)\neq 0$, we have
$$
0 \geq - \phi (x) + \frac{\alpha}{2} \left\{
\max_{\overline{B}_\varepsilon (x)} \phi  +
\min_{\overline{B}_\varepsilon (x)} \phi \right\} + \beta
\kint_{B_\varepsilon (x)} \phi (y) \ud y + o(\varepsilon^2).
$$

\item for every $ \phi \in C^{2}$ such that $\phi$ touches $u$ at the point $x \in \Omega$  from above and $D \phi(x)\neq 0$, we have
$$
0 \leq - \phi (x) + \frac{\alpha}{2} \left\{
\max_{\overline{B}_\varepsilon (x)} \phi  +
\min_{\overline{B}_\varepsilon (x)} \phi \right\} + \beta
\kint_{B_\varepsilon (x)} \phi (y) \ud y
 + o(\varepsilon^2).
$$
\end{enumerate}
\end{definition}

The error term satisfies $o(\eps^2)/\eps^2\to 0$ as $\eps\to 0$.
The main result of the section is the following. 
\begin{theorem}\label{main} Let $1< p \leq \infty$ and $u\in C(\Om)$.  The asymptotic expansion $$
\displaystyle u(x) = \frac{\alpha}{2} \left\{ \max_{\overline{B}_\varepsilon (x)} u +
\min_{\overline{B}_\varepsilon (x)} u
\right\} + \beta \kint_{B_\varepsilon (x)} u (y) \ud y
+ o(\varepsilon^2), \quad \mbox{as }\varepsilon \to 0,
$$ holds in the viscosity sense
if and only if
$$
\Delta_{p}^Nu (x) =0
$$
in the viscosity sense. 
\end{theorem}
%%%%%%%%%%%%%%%%%%%%%%%%%%%%
The proof is given after the next example.

%%%%%%%%%%%%%%%%%%%%%%%%%%%%
\begin{example}
\label{ex:aronsson}
One can show that the 'viscosity sense' is needed above. Indeed, there are solutions to the PDEs that do not satisfy the asymptotic expansion pointwise, only in the viscosity sense.

Set $p=\infty$ and consider Aronsson's function, Figure~\ref{fig:aronsson}, 
\[
u(x,y)= \abs{x}^{4/3}-\abs{y}^{4/3}
\]
near the point $(x,y)=(1,0)$. Aronsson's function is
infinity harmonic in the viscosity sense but it is not of class
$C^{2}$, see Aronsson \cite{aronsson67}. 
\begin{figure}[h]
\includegraphics[scale=0.3]{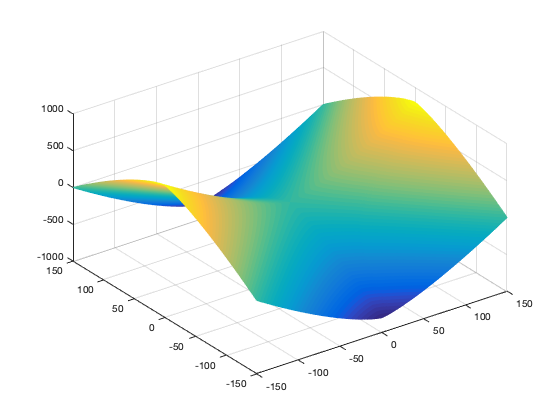}
\caption{Aronsson's function.}
\label{fig:aronsson}
\end{figure}

By Theorem
\ref{main},  $u$ satisfies
$$
u(x) = \frac{1}{2} \left\{ \max_{\overline{B}_\varepsilon
(x)} u +
\min_{\overline{B}_\varepsilon (x)} u
\right\} +   o(\varepsilon^2) \quad \mbox{as }\varepsilon \to 0,
$$
in the viscosity sense of Definition \ref{def.sol.viscosa.asymp}.
However, let us verify that the expansion does not hold in the
classical pointwise sense.

 Clearly, we have
\[
\max_{\overline{B}_\varepsilon (1,0)}
u=u(1+\eps,0)=(1+\eps)^{4/3}.
\]
To find the minimum, we set $x=\eps \cos(\theta)$, $y=\eps
\sin(\theta)$ and solve the equation
\[
\begin{split}
\frac{d}{d\theta}& u(1+ \eps \cos(\theta),\eps
\sin(\theta))\\
&=-\frac{4}{3}(1 + \eps \cos(\theta))^{1/3}\eps
\sin(\theta) - \frac{4}{3} (\eps \sin(\theta))^{1/3} \eps
\cos(\theta)\\
&=0.
\end{split}
\]
By symmetry, we can focus our attention on the solution
$$\theta_\eps=\arccos\left(\frac{\eps-\sqrt{4+\eps^2}}{2}\right).$$ Hence, we obtain
$$
\begin{array}{l}
\displaystyle \min_{\overline{B}_\varepsilon (1,0)} u \displaystyle
=u(1+\eps \cos(\theta_\eps),\eps \sin(\theta_\eps))\\
\qquad  \displaystyle  =
\left(1+\frac{1}{2}\eps\left(\eps-\sqrt{4+\eps^2}\right)\right)^{4/3}
-\left(\eps\sqrt{1-\frac{1}{4}\left(\eps-\sqrt{4+\eps^2}\right)^2}\right)^{4/3}.
\end{array}
$$
We are ready to compute
\[
\lim_{\eps\to 0+}\frac{ \frac{1}{2} \left\{\displaystyle
\max_{\overline{B}_\varepsilon (1,0)} u +
\min_{\overline{B}_\varepsilon (1,0)} u
\right\}-u(1,0)}{\eps^2}=\frac{1}{18}.
\]
But if the asymptotic expansion held in the classical sense, this
limit would have to be zero. 

Later it was observed in \cite{lindqvistm16} for a certain range of $p$ and in  \cite{arroyol16} for $1<p<\infty$ that mean value property holds directly to solutions themselves for $n=2$. To the best of my knowledge the situation is open for $1<p<\infty,\ p\neq 2$ and $n\ge 3$.
\end{example}

\begin{proof}[Proof of Theorem \ref{main}]

In Section \ref{sec:convergence-p},  we derived an asymptotic expansion by  combining asymptotic expansions to the infinity Laplacian ($p=\infty$) and the regular
Laplacian ($p=2$), and the following argument is quite similar to the argument there.
We recall some of the details for the convenience of the reader, and choose a point $x\in \Omega$ and a $C^2$-function  $\phi$. Let $x_1^\varepsilon$  and
$x_2^{\varepsilon}$ be the points where $\phi$ attains its
minimum and maximum  in $\overline{B}_\varepsilon(x)$
respectively; that is,
\[
\phi(x_1^\varepsilon)=\min_{y\in
\overline{B}_\varepsilon(x)}\phi(y)\quad \text{     and    } \quad \phi(x_2^\varepsilon)=\max_{y\in
\overline{B}_\varepsilon(x)}\phi(y).
\]

Assume for the moment that  $2\le p<\infty$ so that   $0\le \a <1$. In (\ref{eq:asymp-p-lap}), we obtained
\begin{equation}
\label{eq:asymp-p-lap-visc}
\begin{split}
&\frac{\alpha}{2}\left\{\max_{y\in \overline{B}_\varepsilon (x)}
\phi (y) + \min_{y\in \overline{B}_\varepsilon(x)} \phi
(y)\right\}+ \beta\kint_{B_\varepsilon (x)} \phi (y)
\ud y -
\phi(x)\\
&\ge \frac{\beta\eps^2}{2(n+2)}\left((p-2)\Big\langle
D^2\phi(x)\left(\frac{x_1^{\varepsilon}-x}{\varepsilon}\right),
\left(\frac{x_1^{\varepsilon}-x}{\varepsilon}\right)\Big\rangle
+
\Delta \phi(x)\right)\\
& \quad +o(\varepsilon^2).
\end{split}
\end{equation}

We are ready to  prove that if the asymptotic mean value formula
holds for $u$, then $u$ is a viscosity solution.  Suppose that function $u$ satisfies the asymptotic expansion
in the viscosity sense according to
Definition~\ref{def.sol.viscosa.asymp}. Then for a test function $\phi$ touching $u$ from below at $x$, we obtain
\[
\begin{split}
 0\geq-\phi (x)&+
\frac{\alpha}{2} \left\{ \max_{\overline{B}_\varepsilon (x)} \phi
+ \min_{\overline{B}_\varepsilon (x)} \phi \right\}+ \beta
\kint_{B_\varepsilon (x)} \phi (y)\ud y + o(\varepsilon^2).
\end{split}
\]
Thus, by \eqref{eq:asymp-p-lap},
\begin{align*}
o(\varepsilon^2)\ge \frac{\beta\eps^2}{2(n+2)}\left((p-2)\Big\langle
D^2\phi(x)\left(\frac{x_1^{\varepsilon}-x}{\varepsilon}\right),
\left(\frac{x_1^{\varepsilon}-x}{\varepsilon}\right)\Big\rangle
+
\Delta \phi(x)\right).
\end{align*}
Thus by dividing the previous expansion by $\eps^2$, and passing to the limit $\varepsilon\to 0$,
we get
\[
\begin{split}
0 &\geq
  \frac{\beta}{2(n+2)}\left((p-2) \Delta_{\infty}^N\phi(x)
+\Delta \phi (x)\right).
\end{split}
\]

In the case $p=\infty$, recall that the derivation of (\ref{eq:asymp-p-lap}) contained
\begin{align}
\label{eq:taylor-min-max-extract}
\half\Bigg\{\max_{y\in \overline{B}_\varepsilon (x)} \phi (y) &+
\min_{y\in \overline{B}_\varepsilon(x)} \phi (y)\Bigg\} - \phi(x)
\\
&
\geq\frac{\eps^2}{2} \left\langle
D^2\phi(x)\left(\frac{x_1^{\varepsilon}-x}{\varepsilon}\right),\left(\frac{x_1^{\varepsilon}-x}{\varepsilon}\right)\right\rangle+o(\varepsilon^2),\nonumber
\end{align}
and the proof is similar as above.

To prove that
$u$ is a viscosity subsolution, we first derive a reverse
inequality to \eqref{eq:asymp-p-lap} by considering the maximum
point $x_2^{\eps}$ of  a test function $\phi$ that
touches $u$ from above. We omit the details.

To prove the converse implication, assume that $u$ is a viscosity
solution. In particular $u$ is a subsolution. Thus for a test function $\phi$ touching $u$ from above
at $x \in \Omega$, $D\phi(x)\neq 0$, it holds
\begin{align*}
(p-2)\Delta_\infty^N \phi(x)+\Delta \phi(x)\ge
0.
\end{align*}
We divide
\eqref{eq:asymp-p-lap} by $\eps^2$, pass to a limit, and use the above subsolution property to deduce
$$
\liminf_{\varepsilon \to 0+} \displaystyle\frac{1}{\eps^2}\left(
\displaystyle- \phi(x) + \frac{\alpha}{2} \left\{
\max_{\overline{B}_\varepsilon (x)} \phi +
\min_{\overline{B}_\varepsilon (x)} \phi \right\} + \beta
\intav_{B_\varepsilon (x)} \phi(y) \ud y
 \right)\geq 0.
$$
The  case $p=\infty$ is analogous. 

Finally, we need to address the case $1< p<2$. Since $\alpha< 0$
we use
maximum point $x_2^{\eps}$ instead of the minimum point $x_{1}^{\eps}$ to get
a version of  (\ref{eq:taylor-min-max-extract}) with $x_{2}^{\varepsilon}$ in
the place  of $x_{1}^{\varepsilon}$, and the inequality reversed. Then multiplying the expansion with $\alpha<0$ before summing it up with the $p=2$ expansion reverses the inequality once more. Thus we get \eqref{eq:asymp-p-lap} except $x_{1}^{\varepsilon}$  is replaced by $x_{2}^{\varepsilon}$. The argument then continues in
the same way as before. 
 \end{proof}

\subsection{Parabolic case}

The parabolic asymptotic mean value characterization for $1< p\le \infty$ was derived in
 \cite{manfredipr10c}. 
\begin{theorem}\label{thm:parab-main-p-lap}
Let $1<p\leq \infty$ and $u\in C(\Om_T)$.
 The asymptotic mean value formula
\begin{equation}
\label{eq:nonintegrated-form-of-p-mvt}
\begin{split}
u(x,t)&= \frac{\alpha}{2}
\left\{ \max_{y\in \ol B_\eps (x)} u\Big(y,t-\frac{\eps^2}{2}\Big) +
\min_{y\in \ol B_\eps (x)} u\Big(y,t-\frac{\eps^2}{2}\Big) \right\}
\\
&\hspace{1 em}+\beta \kint_{B_\eps (x)} u\Big(y,t-\frac{\eps^2}{2}\Big)
\ud y  +o(\eps^2), \quad \mbox{as }\eps \to 0.
\end{split}
\end{equation}
holds for every $(x,t)\in \Om_T$ in the viscosity sense
if and only if $u$ is a viscosity solution to
$$
(n+p)u_t (x,t)=\Delta_p^{N} u (x,t).
$$
\end{theorem}

%%%%%%%%%%%%%%%%%
%%%%%%%%%%
\begin{remark}\label{rem:reduction}
Recall that by Theorem \ref{thm:test-functions-neglect} we can reduce the number of test functions in the definition of viscosity solutions:
Either  $D\phi(x,t)\neq 0$, or  $D\phi(x,t)=0$ and $D^2\phi(x,t)=0$. We use this test function class when we say in the 'viscosity sense' above both with the asymptotic mean value property and with the equation.
In particular, the definition can be rewritten as
\[
\begin{split}
\begin{cases}
 (n+p)\phi_t(x,t)-\Delta_p^N \phi(x,t) \ge 0,& \text{ if }D\phi(x,t)\neq 0,\\
 (n+p)\phi_t(x,t)\ge 0&  \text{ if }D\phi(x,t)=0 \text{ and } D^2\phi(x,t)=0,
\end{cases}
\end{split}
\] 
when touching from below, and analogously from above. This also holds for $p=\infty$ by \cite{juutinenk06}, and for $1<p<2$ by \cite{manfredipr10c}. In the case $p=\infty$, we mean the equation
\begin{align*}
u_t=\Delta_{\infty}^N u.
\end{align*}
\end{remark}
%%%%%%%%%%
%%%%%%%%%%
\begin{remark}
A similar proof that will be used to prove Theorem  \ref{thm:parab-main-p-lap} also gives a characterization of viscosity solutions if
\[
\begin{split}
u(x,t)&= \frac{\alpha}{2}
\aveint{t-\eps^2}{t}\left\{ \max_{y\in \ol B_\eps (x)}
u(y,s) + \min_{y\in \ol B_\eps (x)} u(y,s) \right\} \ud s
\\
&\hspace{4 em}+\beta \aveint{t-\eps^2}{t}\kint_{B_\eps (x)} u(y,s)
\ud y \ud s +o(\eps^2), \quad \textrm{as} \quad \eps \to 0.
\end{split}
\]

If $p=2$, this gives that
\[
\begin{split}
u(x,t)&= \beta \aveint{t-\eps^2}{t}\kint_{B_\eps (x)} u(y,s)
\ud y \ud s +o(\eps^2), \quad \textrm{as} \quad \eps \to 0,
\end{split}
\]
is equivalent to 
\begin{align*}
(n+2)u_t=\Delta u,
\end{align*}
i.e.\ we get an asymptotic  mean value characterization to the solutions of the (rescaled) heat equation as well. This can be compared to the classical mean value formula for $u_t=\Delta u$ that reads as 
 \[
 \begin{split}
 u(x,t)=\frac{1}{4r^n}\int_{E(x,t,r)} u(y,s) \frac{\abs{x-y}^2}{(t-s)^2} \ud y \ud s,
 \end{split}
 \]
 see \cite[Section 2.3.2]{evans10}.
 Here $E(x,t,r)$ denotes heat balls as in Figure~\ref{fig:heatballs} below. They are super level sets for the fundamental solution $\Phi$ with the time direction reversed
 \begin{align*}
E(0,0,r):=\{ (y,s)\in \R^{n+1} \,:\,\Phi(y,-s)\ge\frac{1}{r^n} \}.
\end{align*}
\begin{figure}[h]
\includegraphics[scale=0.3]{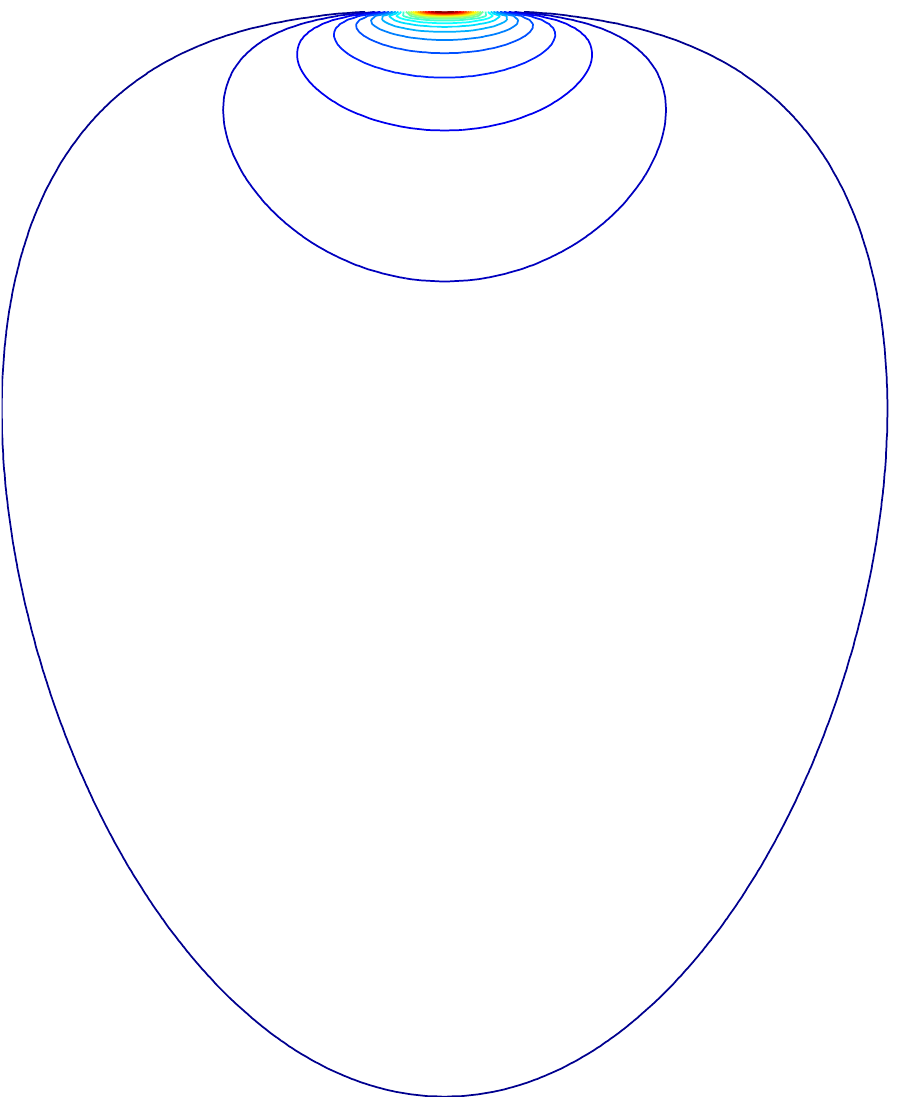}
\caption{Heat balls.}
\label{fig:heatballs}
\end{figure}
\end{remark}
%%%%%%%%%%
 
 Observe that in the following proposition viscosity sense is not used. According to the classical results, a solution to the heat
equation is smooth.
\begin{proposition}\label{thm:heat} Let $u\in C^{\infty}(\Om\times(0,T))$. The asymptotic mean value formula
$$
\label{eq:heatmean}
u(x,t)= \kint_{B_{\eps}(y)}
u\left(y,t-\frac{\eps^2}{2}\right) \ud y +o(\eps^2),
\quad \mbox{as }\eps \to 0.$$
holds for all $(x,t) \in \Om_T$ if and only if
$$
(n+2)u_t (x,t) =\Delta u (x,t)
$$
in $\Omega_T$.
\end{proposition}

\begin{proof} 
 We use the
Taylor expansion in a similar fashion as what was sketched in Section \ref{sec:intro-time-dependent}. Since $u$ is smooth, we get
\begin{align*}
u(x+h,t-s)&=u(x,t)+ \langle  Du(x,t), h \rangle+
\half\langle D^2u(x,t) h,h\rangle\\
&\hspace{1 em}-u_t(x,t)
s+o(\abs{h}^2+\abs{s}).
\end{align*}
Choosing $s=\frac{\eps^2}{2}$ and averaging in space, we get similarly as in (\ref{exp.lapla})
\begin{equation}
\label{eq:parab-taylor-mean}
\begin{split}
&\kint_{B_{\eps}(0)}u(x+h,t-\frac{\eps^2}{2})\ud h \\
&=u(x,t) + \frac{\eps^2}{2(n+2)}\left(\Delta
u(x,t)-(n+2)u_t(x,t)\right) +o(\eps^2).
\end{split}
\end{equation}
If smooth $u$ is a solution to $(n+2)u_t=\Delta u$, then
\eqref{eq:parab-taylor-mean} immediately implies that $u$ satisfies the asymptotic mean value
property, and vice versa. 
\end{proof}

%%%%%%%%%%%%%%%%%%%%%%%
%%%%%%%%%%

\begin{proof}[Proof of Theorem \ref{thm:parab-main-p-lap}]
Proof is quite similar to the proof of Theorem \ref{thm:sol-to-p-parabolic2}, but for the convenience of the reader we recall some of the details. Denoting 
\[
\begin{split}
\phi\Big(x_1^{\eps,t-\eps^2/2},t-\frac{\eps^2}{2}\Big)=
\min_{y\in \ol B_\eps(x)} \phi\Big(y,t-\frac{\eps^2}{2}\Big),
\end{split}
\]
the expansion (\ref{eq:parab-taylor-comb-pointwise}) gives us
\begin{align*}
&\beta \kint_{B_{\eps}(x)}\phi(y,t-\frac{\eps^2}{2})\ud y +\frac{\alpha}{2}
\Bigg\{\max_{y\in \ol B_\eps (x)}  \phi (y,t-\frac{\eps^2}{2}) +
\min_{y\in \ol B_\eps(x)}  \phi (y,t-\frac{\eps^2}{2})\Bigg\} \\
&\ge  \phi(x,t)+ \frac{\beta \eps^2}{2(n+2)}\Bigg\{\Delta
\phi(x,t)-(n+p)\phi_t(x,t) \\
&\hspace{5 em}+(p-2)\left\langle
D^2 \phi(x,t)\frac{x_1^{\eps,t-\eps^2/2}-x}{\eps},\frac{x_1^{\eps,t-\eps^2/2}-x}{\eps}\right\rangle\Bigg\}+o(\eps^2)
\end{align*}
in the case $2\le p< \infty$. In the case $p=\infty$, we have
\begin{align*}
&\frac{1}{2} \Bigg\{\max_{y\in \ol B_\eps (x)}  \phi (y,t-\frac{\eps^2}{2}) +
\min_{y\in \ol B_\eps(x)}  \phi (y,t-\frac{\eps^2}{2})\Bigg\} - \phi(x,t)\\
&\ge \frac{\eps^2}{2}\Bigg\{\left\langle
D^2 \phi(x,t)\frac{x_1^{\eps,t-\eps^2/2}-x}{\eps},\frac{x_1^{\eps,t-\eps^2/2}-x}{\eps}\right\rangle-\phi_t(x,t)\Bigg\}+o(\eps^2).
\end{align*} 
Then if $D\phi(x,t)\neq 0$, dividing the above expansion by $\eps^2$, and utilizing 
\[
\begin{split}
\left\langle
D^2 u(x,t)\frac{x_1^{\eps,t-\eps^2/2}-x}{\eps},\frac{x_1^{\eps,t-\eps^2/2}-x}{\eps}\right\rangle\to \Delta_\infty^{N}u(x,t),
\end{split}
\]
as $\eps\to 0$, we see  by the similar argument as in of the proof of Theorem \ref{main} in the elliptic case that the viscosity solution property implies the asymptotic mean value property, and vice versa. Also the case $1<p<2$ can be treated in the same way as in the elliptic case. 

Let $\phi$ touch
$u$ at $(x,t)$ from below and assume $D \phi(x,t)=0$. According to Remark \ref{rem:reduction}, we may also assume that
$D^2\phi(x,t)=0$, and thus the Taylor expansion reads as 
\[
\begin{split}
\phi(y,t-\frac{\eps^2}{2})-\phi(x,t)=-\phi_t(x,t)\frac{\eps^2}{2}+o(\eps^2)
\end{split}
\]
for $y\in \ol B_{\eps}(x)$.
It follows that
\[
\begin{split}
 &\beta \kint_{B_{\eps}(x)}\Big(\phi(y,t-\frac{\eps^2}{2})-\phi(x,t)\Big)\ud y\\
&+\frac{\alpha}{2}
\Bigg\{\max_{y\in \ol B_\eps (x)}  \Big(\phi (y,t-\frac{\eps^2}{2}) -\phi(x,t)\Big)+
\min_{y\in \ol B_\eps(x)}  \Big(\phi (y,t-\frac{\eps^2}{2})-\phi(x,t)\Big)\Bigg\} \\
&=-\phi_t(x,t)\frac{\eps^2}{2} +o(\eps^2).
\end{split}
\]
Dividing by $\eps^2$, and passing to the limit, we observe that the asymptotic mean value formula holds  in the viscosity sense if and only if the equation holds in the viscosity sense.
\end{proof}

%%%%%%%%%%%%%%%%%%%%%%%%%%%%%%%

\section{Further regularity methods}
\label{sec:other-methods}

We already saw the cancellation method in Section \ref{sec:Lip-regularity} for the asymptotic Lipschitz regularity of value functions. It is a nontrivial task to extend that method to a more general class of problems where the sharp cancellation used in the proof breaks down. In this section, we review some regularity methods applicable with more general assumptions.

 \subsection{Couplings of dynamic programming principles}
\label{sec:intuition}
This method for $p$-Laplacian was first devised in \cite{luirop18} and then in the parabolic case in \cite{parviainenr16}. 
Although the proofs in the above papers are mainly written without game methods, the intuition comes from the stochastic games. 
We illustrate the method in the case of the pure tug-of-war and pure random walk, but the method also works for example to the games with space dependent probabilities and to the games corresponding to the full range $1<p\le \infty$.

\noindent {\bf Tug-of-war:} Let
\begin{align*}
u_\varepsilon(x)=\frac{1}{2}\sup_{B_\varepsilon(x)}u_\varepsilon+\frac{1}{2}\inf_{B_\varepsilon(x)}u_\varepsilon.
\end{align*}
Then
 $$
 G(x,z):=u_{\eps}(x)-u_{\eps}(z)
 $$ can be written as a solution of a certain  DPP in $\R^{2n}$:
 \begin{align}\label{eq:DPP-2}
				G(x,z)=u_\eps(x)-u_\eps(z)&=\frac{1}{2}\,\big(\,\sup_{B_\eps(x)}u_\eps+\inf_{B_\eps(x)}u_{\eps}-\sup_{B_\eps(z)}u_\eps-\inf_{B_\eps(z)}u_{\eps}\,\big)\nonumber \\
				&=\frac{1}{2}\sup_{B_{\eps}(x)\times B_{\eps}(z)}G\,+\,\frac{1}{2}\inf_{B_{\eps}(x)\times B_{\eps}(z)}G\,.
\end{align}
\idea In this way the question about the regularity of $u_{\eps}$
is converted into a question about the \text{absolute size} of a solution $G$ of \eqref{eq:DPP-2} in $\Omega\times\Omega\subset \R^{2n}$.

Next we explain the idea of getting a bound for $|G(x,z)|$ via a stochastic game in $\R^{2n}$. We utilize the observation that $G=0$ in the diagonal set $$T:\,=\{(x,z): x=z\}.$$
The rules of the game are as follows: two game tokens are placed in $\Omega$. Two players play the game so that at each turn the players have an equal chance to win the turn. If a player wins the turn when the game tokens are at $x_k$ and $z_k$ respectively, he can move the game token at $x_k$ to any point in $B_{\eps}(x_k)$ and the game token at $z_k$ anywhere in $B_{\eps}(z_k)$. We use the following stopping rules and payoffs\sloppy
\begin{enumerate}
\item game tokens have the same position, and payoff is zero
\item one of the game tokens is placed outside $B_1\subset \Omega$, and  payoff is  $2\sup |u_{\eps}|$
\end{enumerate}
We try to minimize the payoff and the opponent tries to maximize it. 

\idea We try to pull the game tokens to the same position before the opponent succeeds in moving one of the game tokens outside $B_1$. The value of this game should evidently be
larger than or equal to $|G|$.

Usually when starting from a game in $\Rn$ one could derive several different games in $\R^{2n}$, and we need to choose a  game that is suitable for our purposes. In stochastics or optimal mass transport language, we choose couplings of probability measures on $\Rn$ in such a way that we get a probability measure on $\R^{2n}$ having the original measures as marginals.

A more analytic way of looking at details of the method is that we find a suitable supersolution $f$ to the multidimensional DPP that yields the desired bound for the solution $G$ after establishing a comparison principle.
We now sketch the concrete steps of the method. We take a comparison function
\begin{align*}
f(x,z)=C\abs{x-z}^{\delta}+\ldots
\end{align*} 
where we omitted the second localizing term for simplicity, and also a correction needed if $x$ and $z$ are close to each other. In any case, if we can establish a comparison between $u(x)-u(x)=G(x,z)$ and $f(x,z)$, then Hölder regularity will follow.  To establish the comparison, thriving for a contradiction assume that 
\begin{equation}\label{vastaoletus}
M:=\sup_{(x',z')\in B_1\times B_1}(u_{\eps}(x')-u_{\eps}(z')-f(x',z'))\,
\end{equation}
is too big. We choose $(x,z)\in B_1\times B_1$\footnote{We assume for simplicity, that we attain all $\inf$s and $\sup$s in this informal discussion to avoid adding small constants here and there.} such that
\begin{align*}
%\label{eq:sup-points}
u_{\eps}(x)-u_{\eps}(z)-f(x,z)= M.
\end{align*}
Using this and the DPP for the tug-of-war, we get
\begin{equation}
\label{eq:tgw-main-est}
\begin{split}
M&= u_{\eps}(x)-u_{\eps}(z)-f(x,z)\\ 
&=\frac{1}{2}\sup_{B_{\eps}(x)}u_{\eps}
+\frac{1}{2}\inf_{B_{\eps}(x)}u_{\eps}
-\bigg(\frac{1}{2}\sup_{B_{\eps}(z)}u
+\frac{1}{2}\inf_{B_{\eps}(z)}u\bigg)
-f(x,z)\\
&=
\frac{1}{2}\bigg(\sup_{B_{\eps}(x)}u_{\eps}-\inf_{B_{\eps}(z)}u_{\eps}+\inf_{B_{\eps}(x)}u_{\eps}-\sup_{B_{\eps}(z)}u_{\eps}\,\bigg)
-f(x,z)\\
&=\,I_1-f(x,z).
\end{split}
\end{equation}

Next we estimate $I_1$.
Suppose that $x_0\in B_{\eps}(x),z_0\in B_{\eps}(z)$ are such that $\sup_{B_{\eps}(x)}u_{\eps}= u_{\eps}(x_0)$ and $\inf_{B_{\eps}(z)} u_{\eps}=u_{\eps}(z_0)$. Then
\begin{align*}
%\label{eq:sup-inf}
\sup_{B_{\eps}(x)}u_{\eps}-\inf_{B_{\eps}(z)}u_{\eps}=&u_{\eps}(x_0)-u_{\eps}(z_0)\nonumber\\
=&u_{\eps}(x_0)-u_{\eps}(z_0)-f(x_0,z_0)+f(x_0,z_0)\nonumber\\
\leq &M+f(x_0,z_0)\nonumber\\
\leq& M+\sup_{B_{\eps}(x)\times B_{\eps}(z)}f\,.
\end{align*}
Reselect next $x_0\in B_{\eps}(x)$, $z_0\in B_{\eps}(z)$ such that $f(x_0,z_0)=\inf_{B_{\eps}(x)\times B_{\eps}(z)}f$. It follows that
\begin{align*}
\inf_{B_{\eps}(x)}u_{\eps}-\sup_{B_{\eps}(z)}u_{\eps}&\leq  u_{\eps}(x_0)-u_{\eps}(z_0)\\
&=\, u_{\eps}(x_0)-u_{\eps}(z_0)-f(x_0,z_0)+f(x_0,z_0)\\
&\leq  \,M+f(x_0,z_0)\,\\
&=  M+\inf_{B_{\eps}(x)\times B_{\eps}(z)}f.
\end{align*}
Combining the estimates, we get
\[
\begin{split}
I_1 \leq M+\half \bigg(\sup_{B_{\eps}(x)\times B_{\eps}(z)}f \,+\,\inf_{B_{\eps}(x)\times B_{\eps}(z)}f\,\bigg).
\end{split}
\]
Recalling (\ref{eq:tgw-main-est}), we end up with 
\begin{align*}
M\le M+\half \bigg(\sup_{B_{\eps}(x)\times B_{\eps}(z)}f \,+\,\inf_{B_{\eps}(x)\times B_{\eps}(z)}f\,\bigg)-f(x,z).
\end{align*}
Thus for the desired contradiction it suffices to show that 
\begin{align*}
f(x,z)&>\frac{1}{2}\bigg(\sup_{B_{\eps}(x)\times B_{\eps}(z)}f \,+\,\inf_{B_{\eps}(x)\times B_{\eps}(z)}f\,\bigg).
\end{align*}
To establish the last inequality we use the precise form of $f$ and suitable choices (that can be guessed based on the game strategies) to estimate the right hand side.

\noindent {\bf Random walk:} To give another example, consider the random walk. We choose a mirror point coupling when points are far away from each other. Thus 
\begin{align*}
G(x,z)&=u_{\eps}(x)-u_{\eps}(z)=\kint_{B_{\eps}(0)}u_{\eps}(x+h) \ud h-\kint_{B_{\eps}(0)}u_{\eps}(z+h) \ud h\\
 &=\kint_{B_{\eps}(0)}u_{\eps}(x+h)- u_{\eps}(z+P_{x,z}(h)) \ud h\\
 &=\kint_{B_{\eps}(0)}G(x+h, z+P_{x,z}(h)) \ud h
\end{align*}
where $P_{x,z}(h)$ is a mirror point of $h$ with respect to $V^\perp=\operatorname{span}(x-z)^\perp$, as in Figure~\ref{fig:mirror}.
\begin{figure}[h]
\includegraphics[scale=0.9]{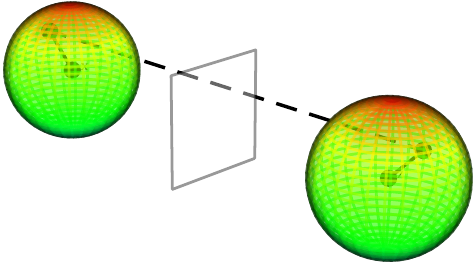}
\caption{Mirror point coupling.}
\label{fig:mirror}
\end{figure}

\idea As indicated at the beginning of the section,  we try to get to the target $x=z$ to get a favorable bound for $G$. 
In stochastic terms,  the mirror point coupling is a good coupling when aiming at $x=z$ in the course of the process. In contrast, choosing for example the same vector $h$ instead of the mirror point $P_{x,z}(h)$ would be a bad choice with this aim.

Again, with a more analytic approach, to prove a comparison between $G$ and $f$, we make a similar counter assumption as above and estimate
\begin{align*}
M&= u_{\eps}(x)-u_{\eps}(z)-f(x,z)\\ 
&=\bigg(\intav_{B_{\eps}(x)}u_{\eps}(y)\,dy\,-\intav_{B_{\eps}(z)}u_{\eps}(y)\,dy\,\bigg)
-f(x,z)\\
&=\kint_{B_{\eps}(0)}\Big(u_{\eps}(x+h)-u_{\eps}(z+P_{x,z}(h))-f(x+h,z+P_{x,z}(h))\Big)\,dh\\
&\,\,\,\,\,\,+\kint_{B_{\eps}(0)}f(x+h,z+P_{x,z}(h))\,dh\,-f(x,z)\\
&\leq M+\kint_{B_{\eps}(0)}f(x+h,z+P_{x,z}(h))\,dh-f(x,z).
\end{align*}
Thus showing 
\begin{align*}
\kint_{B_{\eps}(0)}f(x+h,z+P_{x,z}(h))\,dh\,<f(x,z)
\end{align*}
provides the contradiction in this case. 
It turns out that the mirror point coupling is a favorable choice when establishing the above inequality using the explicit form of $f$. 

A parabolic version of the above approach is in \cite{parviainenr16}.

The argument sketched in this section has connections to the method of couplings of stochastic processes dating back to the 1986 paper of Lindvall and Rogers \cite{lindvallr86}. 
In a seemingly independent thread in the theory of viscosity solutions, the coupling method is related to the doubling of variables  and the theorem on sums as pointed out in \cite{porrettap13}; and also to the Ishii-Lions PDE regularity method introduced in \cite{ishiil90}. 

Even if the Ishii-Lions method is not required for the above approach, a simple presentation of the method might be of independent interest, and thus we give it here.
For simplicity we look at the Laplacian, but the method generalizes to a class of operators that are uniformly elliptic. For the definition of viscosity solutions, recall Definition \ref{eq:viscosity-sol} with $p=2$. 

\begin{theorem}[Ishii-Lions]
Let $u\in C(\ol B_4(0))$ be a continuous viscosity solution to $\Delta u=0$.
Then for $\a\in (0,1)$ there is $C>0$ such that
\[
\begin{split}
\abs{u(x)-u(y)}\le C\abs{x-y}^{\a} \text{ for all } x,y\in B_{1/4}(0).
\end{split}
\]

\end{theorem}
%%%%%%%%%%%%%%%%%%%%%%%%
%%%%%%%%%%%%%%%%%%%%%%%%
\begin{proof}
By considering $v:=(u-\inf_{B_1} u)/\lambda$, $\lambda :=\sup u_{B_1}-\inf_{B_1} u$, we may assume that
\[
\begin{split}
0\le u \le 1.
\end{split}
\]
Choose $z\in B_{1/4}$ and set 
 $$
 f(x,y):=\vp(x,y)+2\abs{x-z}^2:=C\abs{x-y}^\a+2\abs{x-z}^2
 .$$
  We want to show that $u(x)-u(y)-f(x,y)\le 0$, and thriving for a contradiction  suppose that there is $\theta>0$ and $x,y\in B_1$ such that
 \[
\begin{split}
u(x)-u(y)-f(x,y)=\sup_{(x',y')\in \ol B_1\times \ol B_1}(u(x')-u(y')-f(x',y'))&= 
 \theta>0.
\end{split} 
\]
We immediately observe that $x\neq y$ by the counter assumption.
Also $(x,y)\notin \partial (B_1\times B_1)$ since if $(x,y)\in \partial (B_1\times B_1)$
\[
\begin{split}
C\abs{x-y}^\a+2\abs{x-z}^2\ge 1,
\end{split}
\]
for $C$ large enough.

  Thus we may use the theorem on sums (see for example \cite[Lemma 3.6]{koike04}, or \cite[Theorem 3.3.2]{giga06}) to obtain symmetric matrices $X$ and $Y$ such that
\[
\begin{split}
(D_x\vp(x,y),X)\in \ol J^{2,+}(u(x)-2\abs{x-z}^2),\quad (-D_y\vp(x,y),Y)\in \ol J^{2,-}u(y)
\end{split}
\]
i.e.\ by \cite[Proposition 2.7]{koike04}
\[
\begin{split}
(D_x\vp(x,y)+4(x-z),X+4I)\in \ol J^{2,+}u(x),\ (-D_y\vp(x,y),Y)\in \ol J^{2,-}u(y),
\end{split}
\]
with the estimate
\begin{equation}
\label{eq:thm-sum-est}
\begin{split}
\begin{pmatrix}
X&0\\
0&-Y 
\end{pmatrix}\le D^2 \vp(x,y)+\frac{1}{\mu} (D^2 \vp(x,y))^2. 
\end{split}
\end{equation}
Now
\[
\begin{split}
D\vp(x,y)&=C\a\abs{x-y}^{\a-2}(x-y,-(x-y)),\\
 D^2\vp(x,y)&=\begin{pmatrix}
M&-M\\
-M&M 
\end{pmatrix},
\end{split}
\]
with 
\begin{align*}
M=C\a\abs{x-y}^{\a-2}\Big((\a-2)\frac{x-y}{\abs{x-y}}\otimes \frac{x-y}{\abs{x-y}}+I\Big),
\end{align*}
where
$x\otimes y$ denotes a matrix with entries $x_iy_j$, and $I$ is an identity matrix.
Further,
\[
\begin{split}
(D^2\vp(x,y))^2=2\begin{pmatrix}
M^2&-M^2\\
-M^2&M^2 
\end{pmatrix}.
\end{split}
\]
where 
\[
\begin{split}
&M^2=C^2\a^2\abs{x-y}^{2(\a-2)} \\
\\&\hspace{4 em}\cdot\Big((\a-2)^2\frac{x-y}{\abs{x-y}}\otimes \frac{x-y}{\abs{x-y}}+2(\a-2)\frac{x-y}{\abs{x-y}}\otimes \frac{x-y}{\abs{x-y}}+I\Big)\\
&=C^2\a^2\abs{x-y}^{2(\a-2)} \Big(\underbrace{(\a^2-4\a+4+2\a-4)}_{=\a(\a-2)}\frac{x-y}{\abs{x-y}}\otimes \frac{x-y}{\abs{x-y}}+I\Big).
\end{split}
\]
%%%%%%%%%%%%%%%%%
%%%%%%%%%
First, all the eigenvalues of $X-Y$ are nonpositive since the right hand side of (\ref{eq:thm-sum-est}) annihilates vectors of the form 
$
\begin{pmatrix}
\xi\\
\xi
\end{pmatrix}
$.
Moreover  by using
\begin{align*} 
\begin{pmatrix}
\xi\\
0 
\end{pmatrix}\quad \text{ and }\quad \begin{pmatrix}
0\\
\xi 
\end{pmatrix}
\end{align*}
in (\ref{eq:thm-sum-est}) respectively, we get 
\begin{equation}
\label{eq:sep-ests}
\begin{split}
\langle X\xi, \xi\rangle &\le \big\langle (M+\frac2{\mu}M^2)\xi ,  \xi\big\rangle,\\
-\langle Y\xi, \xi\rangle &\le \big\langle (M+\frac2{\mu}M^2)\xi ,  \xi\big\rangle.
\end{split}
\end{equation}
Next we observe that
\[
\begin{split}
2\left\langle (M+\frac{2}{\mu} M^2 )\frac{x-y}{\abs{x-y}},\frac{x-y}{\abs{x-y}} \right\rangle&
=2C\a\abs{x-y}^{\a-2}(\a-2+1)\\
&\hspace{1 em}+\frac{4}{\mu} C^2\a^2\abs{x-y}^{2(\a-2)}(\a(\a-2)+1)\\
&\le  C\a\abs{x-y}^{\a-2}(\a-1)
\end{split}
\]
by choosing $\mu$ such that 
\[
\begin{split}
\frac{4}{\mu} C^2\a^2\abs{x-y}^{2(\a-2)}(\a(\a-2)+1)<- C\a\abs{x-y}^{\a-2}(\a-2+1).
\end{split}
\] 
This together with (\ref{eq:sep-ests}) implies that one of the eigenvalues of $X-Y$ will have to be smaller than $C\a\abs{x-y}^{\a-2}(\a-1)$. This together with nonpositivity of eigenvalues implies
\[
\begin{split}
0\le \tr(X+4I)-\tr(Y)&\le 4n+ C\a\abs{x-y}^{\a-2}(\a-1)<0,
\end{split}
\]
for large enough $C=C(n,\a)$. This contradiction completes the proof of 
\begin{align*}
u(x)-u(y)\le C\abs{x-y}^\a+2\abs{x-z}^2.
\end{align*}
Suppose then that we want to obtain $\abs{u(\ol x)-u(\ol y)}\le C\abs{\ol x-\ol y}^\a$ for some given $\ol x,\ol y\in B_{1/4}$. We may always choose the points so that $u(\ol x)\ge u(\ol y)$, and we may select $z=\ol x$. Thus the claim follows from the previous estimate.
\end{proof}
%%%%%%%%%%%%

\subsection{Krylov-Safonov regularity approach for discrete stochastic processes}

The celebrated Krylov-Safonov \cite{krylovs79} H\"older estimate is one of the key results in the theory of nondivergence form elliptic partial differential equations with  bounded and measurable coefficients with no further regularity assumptions.  Later Trudinger \cite{trudinger80} devised an analytic proof for the Krylov-Safonov result.

In \cite{arroyobp, arroyobp2}, a H\"older regularity is established to expectations of stochastic processes with bounded and measurable increments, in other words to functions satisfying the dynamic programming principle
\begin{align}
\label{eq:dpp-intro}
u (x) =\alpha  \int_{\R^N} u_{\eps}(x+\eps z) \,d\nu_x(z)+\frac{\beta}{\abs{B_\eps}}\int_{B_\eps(x)} u_{\eps}(y)\,dy
+\eps^2 f(x).
\end{align}	
Here $f$ is a Borel measurable bounded function, $\nu_x$ is a symmetric probability measure for each $x$ with support in $B_\Lambda,\ \Lambda\geq 1$ and certain measurability conditions, $\eps>0$, and $\alpha+\beta=1,\ \alpha\ge 0,\ \beta>0$. Moreover, the result generalizes to Pucci type extremal operators and conditions of the form 
\begin{align*}
\mathcal L_\varepsilon^+ u\ge -|f|,\quad \mathcal L_\varepsilon^- u\le  |f|, 
\end{align*}
where $\mathcal L_\varepsilon^+, \mathcal L_\varepsilon^-$ are Pucci type extremal operators related to (\ref{eq:dpp-intro}). As a consequence, the results immediately cover for example tug-of-war with noise type stochastic games, see \cite[Remark 7.6]{arroyobp}.

\idea Original idea of Krylov and Safonov was to use a probabilistic argument for PDEs. This uses a fact that there is an underlying continuous time process and a solution $u$ can be presented as an expectation. The key result is to establish that one can reach any set of positive measure with a positive probability before exiting a bigger set.

With this key result at our disposal, the De Giorgi oscillation estimate follows in a straightforward manner.
To be more precise, De Giorgi type oscillation estimate roughly states the following in a simplified setting:  
if $u$ is a (sub)solution with $u\leq 1$ in a suitable bigger set and
\[
|B_{R}\cap \{u\leq 0\}|\geq \theta |B_R|,
\]
for some $\theta>0$, then there exists $\eta>0$ such that
\[
\sup_{B_R} u \leq 1-\eta.
\]
Recalling the key result, we can reach the level set in the assumption above with a positive probability and use this in obtaining the improved oscillation estimate above.
The H\"older estimate then follows from the De Giorgi oscillation lemma after an iteration. 

The idea of proving the key result of reaching any set of positive measure with a positive probability is then as follows. 
If the portion of the set we want to reach has high enough density in a cube, then we can utilize the Alexandrov-Bakelman-Pucci (ABP) PDE-estimate with a characteristic function of the set as a source term in the corresponding PDE to estimate the desired probability. Then it remains to use the Calder\'on-Zygmund decomposition to find such cubes that the density condition is always satisfied, and showing that we reach these cubes with positive probability.
This proof clearly demonstrates an interplay between stochastics and PDEs.

The above sketch utilizes the fact that the PDE and the process contains information in all the scales. Indeed, the Calder\'on-Zygmund decomposition contains arbitrarily small cubes and rescaling is freely applied. In the setting of (\ref{eq:dpp-intro}), we have a discrete process, and the step size sets a natural limit for the scale. This limitation has some crucial effects as we have already seen: values can even be discontinuous, and our estimates are asymptotic. Thus the discrete step size needs to be carefully taken into account in the estimates. 

\section{Open problems and comments}

The following problems are open to the best of my knowledge.
\begin{enumerate}
\item Uniqueness to normalized equations. Uniqueness to $\Delta^N_p u=f$ $1<p< \infty$ whenever $f$ changes sign or even $f\ge 0$. For $\Delta^N_\infty u=f$ there is  a counterexample in \cite{peresssw09} for sign changing $f$. The case $\Delta^N_\infty u=f$ for $f\ge 0$ is also open. 
\item A related question to the previous one: show a uniqueness for $\Delta_{p(x)}^N u=0$ (see a preprint version of \cite{juutinenlp10})? 
\item PDE regularity method. Quite often there is a PDE method related to the stochastic method. Is there a corresponding PDE method to the stochastic regularity method introduced in Section \ref{sec:Lip-regularity}?  
\item The asymptotic mean value property. Does the asymptotic mean value property hold directly to solutions themselves when $p<\infty$? It holds in the viscosity sense as we showed above. In the plane $n=2$ it is known to hold directly to solutions when $1<p<\infty$, see \cite{arroyol16} and \cite{lindqvistm16}. For $p=\infty$ this is false as we saw in Example \ref{ex:aronsson}. Naturally $p=2$ is also known.
\item Games can easily be defined on much more general spaces (suitable metric measure spaces, graphs etc.). How much does the game theoretic approach provide new theory in this context?  The case of a Heisenberg group is discussed for example in \cite{lewickamr20}, and the case of graphs in the references mentioned below in the context of machine learning.
\end{enumerate}
In these lecture notes we concentrated on the case $2\le p<\infty$ in the context of game theory. The reason was to avoid adding extra complications. Needless to say, there are numerous generalizations: for some of those for example \cite{blancr19b} and \cite{lewicka20} can be consulted. This branch of game theory has also provided novel insight into the theory of PDEs from the beginning, see for example \cite{armstrongs10}. We finish by mentioning connections to numerical methods in PDEs \cite{oberman13}, nonlocal theory \cite{bjorklandcf12}, and applications to machine learning \cite{calder19b,caldergl22}.

%%%%%%%%%%%%%%%%%%%%%%%%%%%%%%%%%%%%%%%%%%%%%%%%%%%%%%%%%%%%%%%%
%%%%%%%%%%%%%%%%%%%%%%%%%%%%%%%%%%%%%%%%%%%%%%%%%%%%%%%%%%%%%%%%
%%%%%%%%%%%%%%%%%%%%%%%%%%%%%%%%%%%%%%%%%%%%%%%%%%%%%%%%%%%%%%%%

%\bibliography{citations2}
%\bibliographystyle{alpha}

\def\cprime{$'$} \def\cprime{$'$} \def\cprime{$'$}

\end{document}